\documentclass{article}
\usepackage[lang = british]{ems-lem} %% change to `american' if you use American English

\usepackage{bbm,fixmath}

\renewcommand{\ldots}{\hspace{0.9pt}.\hspace{0.3pt}.\hspace{0.3pt}.\hspace{1.6pt}}

\renewcommand{\C}[1]{{\protect\mathcal{#1}}}
\renewcommand{\ge}{\geqslant}
\renewcommand{\le}{\leqslant}

\theoremstyle{plain}
\newtheorem{theorem}{Theorem}[section]
\newtheorem{lemma}[theorem]{Lemma}
\newtheorem{proposition}[theorem]{Proposition}

\newtheorem{claim}[theorem]{Claim}
\newcommand{\cqed}{\nolinebreak\mbox{\hspace{5 true pt}%
		\rule[-0.85 true pt]{2.0 true pt}{8.1 true pt}}}
\newcommand{\bcpf}{\smallskip\noindent{\it Proof of Claim.} }
\newcommand{\ecpf}{\cqed \medskip}

\theoremstyle{definition}
\newtheorem{remark}[theorem]{Remark}

\newcommand{\beq}[1]{\begin{equation}\label{#1}}
\newcommand{\eeq}{\end{equation}}

\newcommand{\hide}[1]{}

\newcommand{\B}[1]{{\mathbold #1}}
\newcommand{\I}[1]{{\mathbbm #1}}
\newcommand{\V}[1]{\mathbold{#1}}

\newcommand{\e}{\varepsilon}
\newcommand{\ga}[2]{(\gamma_{#1})_{#2}}
\newcommand{\Ad}[1]{\mathrm{Ad}(#1)}
\newcommand{\Laplace}{\mathop{\mathrm \Delta}}
\newcommand{\SO}{\mathrm{SO}}
\newcommand{\Zero}{{\mathbf 0}}
\newcommand{\Sd}{{\I S^{d-1}}}
\newcommand{\affdim}{\dim_{\mathrm{aff}}}

\newcommand{\Borel}{\C B}
\newcommand{\Frac}{\C F}

\newcommand{\Ng}{{\C N}}
\newcommand{\actson}{{\curvearrowright}}

\newcommand{\dd}{\,\mathrm{d}}
\newcommand{\ha}[1]{\cite[#1]{Hassett07itag}}
\newcommand{\cu}[1]{\cite[#1]{Cutkosky18iag}}
\newcommand{\gr}[1]{\cite[#1]{Groemer96gafssh}}
\newcommand{\kn}[1]{\cite[#1]{Knapp:lgbi}}
\newcommand{\TW}[1]{\cite[#1]{TomkowiczWagon:btp}}

\numberwithin{equation}{section}

\begin{document}

%------
% Insert the title of your paper and (if necessary)
% a short title for the running head.
%------
\title{Divisibility of Spheres with Measurable Pieces}
\titlemark{Divisibility of Spheres with Measurable Pieces}

%------
% Insert full names of the authors.
% Add further authors as follows:
%  \emsauthor{2}{}{}
%  \emsauthor{3}{}{}
% etc.
% Abbreviate first names for the running head.
%------
\emsauthor{1}{Clinton T.\ \surname{Conley}}{C.~T.~\surname{Conley}}
\emsauthor{2}{Jan \surname{Greb\'\i k}}{J.~\surname{Greb\'\i k}}
\emsauthor{3}{Oleg \surname{Pikhurko}}{O.~\surname{Pikhurko}}

%------
% Use \authormark if the list of authors is too
% long for the running head: \authormark{A.~\surname{Doe} et al.}
%------

%------
% Add one \emsaffil and one \email for each author.
%------
\emsaffil{1}{Department of Mathematical Sciences,
Carnegie Mellon University,
Pittsburgh, PA 15213, USA}
\emsaffil{2}{Mathematics Institute, University of Warwick,
Coventry CV4 7AL, UK}
\emsaffil{3}{Mathematics Institute and DIMAP,
 %Centre for Discrete Mathematics and its Applications (DIMAP)\\
University of Warwick,
Coventry CV4~7AL, UK}

%------
% Add MSC 2020 codes according to www.ams.org/msc/msc2020.html.
% Secondary codes (in square brackets) are optional.
%------
\classification[%
33C55, %  	Spherical harmonics
%42C10, %  	Fourier series in special orthogonal functions (Legendre polynomials, Walsh functions, etc.)
43A90, % 	Harmonic analysis and spherical functionsXXxXX
57M60%  	Group actions on manifolds and cell complexes in low dimensions
]{03E15, %  	Descriptive set theory
28A05%  	Classes of sets (Borel fields, $\sigma$-rings, etc.), measurable sets, Suslin sets, analytic setsYYyYY
% 54H05%  	Descriptive set theory (topological aspects of Borel, analytic, projective, etc. sets)
}

%------
% Add a list of keywords.
%------

%------
% Insert your abstract.
%------
\begin{abstract}
 For an $r$-tuple $(\gamma_1,\ldots,\gamma_r)$ of special orthogonal $d\times d$ matrices, we say that the Euclidean $(d-1)$-dimensional sphere $\I S^{d-1}$
is \emph{$(\gamma_1,\ldots,\gamma_r)$-divisible} if there is a subset $A\subseteq \I S^{d-1}$ such that its translations by the rotations $\gamma_1,\ldots,\gamma_r$ partition the sphere. Motivated by some old open questions of Mycielski and Wagon, we investigate the version of this notion where the set $A$ has to be measurable with respect to the spherical measure. Our main result shows that measurable divisibility is impossible for a ``generic'' (in various meanings) $r$-tuple of rotations. This is in stark contrast to the recent result of Conley, Marks and Unger which implies that, for every ``generic'' $r$-tuple, divisibility is possible with parts that have the property of Baire. 
\end{abstract}

\maketitle

\section{Introduction}

%Let $d\ge 2$ be a given integer. 
Let $\SO(d)$ denote the group of  \emph{special orthogonal} $d\times d$  matrices, that is, real $d\times d$ matrices $M$ such that the determinant of $M$ is 1 and $M^TM=I_d$, where $I_d$ denotes the identity $d\times d$ matrix.  
The elements of this group are naturally identified with orientation-preserving 
isometries of the Euclidean unit sphere 
$$
\Sd:=\{\B x\in \I R^d\mid\|\B x\|_2=1\},
$$ and we will often refer to them as \emph{rotations}. 

For an $r$-tuple $\B\gamma=(\gamma_1,\ldots,\gamma_r)\in\SO(d)^r$, we say that $\Sd$ is \emph{$\B\gamma$-divisible} (or admits a \emph{$\B\gamma$-division})
if there is $A\subseteq \Sd$ such that its translates $\gamma_1.A,\ldots,\gamma_r.A$ 
partition $\Sd$ (that is, for every $\B x\in\I S^{d-1}$ there are unique $\B y\in A$ and $i\in [r]$ such that $\B x=\gamma_i.\B y$, where we denote $[r]:=\{1,\ldots,r\}$). 
Of course, a set $A$ works for $\B\gamma$ if and only if $\gamma_r.A$ works for $\B\beta:=(\gamma_1\gamma_r^{-1},\ldots,\gamma_{r-1}\gamma_r^{-1},I_d)$.
%, where $I_d\in \SO(d)$ denotes the $d\times d$ identity matrix. Thus $\B\gamma$-divisibility coincides with $\B\beta$-divisibility.
However,  we do not normally assume that any particular rotation is the identity, mostly for the notational convenience so that all indices can be treated uniformly. 
%So, when it is convenient, we may assume that $\gamma_r=I_d$ (namely we do this in Section~\ref{se:d=2}).

%we can always assume that $\gamma_1$ is the identity matrix $\Im$ by

We say that $\Sd$ is \emph{$r$-divisible} if there is an $r$-tuple $\B\gamma\in\SO(d)^r$ such that $\Sd$ is $\B\gamma$-divisible (or, in other words, if we can partition $\Sd$ into $r$ congruent pieces).
The integer pairs $d,r\ge 2$ such that $\Sd$ is $r$-divisible have been completely classified (see e.g.\ Theorem~6.6 in the book by Tomkowicz and Wagon~\cite{TomkowiczWagon:btp}). Namely, the only pairs when the answer is in the negative are when $r=2$ and $d$ is odd. In this case, the impossibility of any $(\gamma_1,\gamma_2)$-division follows from considering a fixed point $\B x\in\Sd$ of $\gamma_1^{-1}\gamma_2$ which exists as the dimension $d-1$ of the sphere is even. (Indeed, no set $A$ can work here: the translates  $\gamma_1.A$ and $\gamma_2.A$ intersect if $\B x\in A$ and do not cover $\gamma_1.\B x$ if~$\B x\not\in A$.)
On the other hand, the case of $d=2$ is trivial (e.g.\ one can take the $r$ rotations of the circle $\I S^1$ by multiples of the angle $2\pi/r$) while the first published solution for $\I S^2$ seems to be by Robinson~\cite[Page 254]{Robinson47fm}. 
%Interestingly, the main motivation behind Robinson's paper~\cite{Robinson47fm} was to show that the smallest possible number of pieces in the Banach-Tarski Paradox~\cite{BanachTarski24} is~5. 
%, as a direct consequence of his more general theorem on Page~252. 
Furthermore, the $r$-divisibility for $\I S^{d-1}$ easily implies the $r$-divisibility of $\I S^{d+1}$, see e.g.\ the proof of Theorem~6.6 in~\cite{TomkowiczWagon:btp} or Lemma~\ref{lm:extend} here.

Mycielski~\cite{Mycielski55} showed that there is a subset $A\subseteq \I S^2$ such that for every integer $r\ge 3$ there are $\gamma_1,\ldots,\gamma_r$ with $\gamma_1.A,\ldots,\gamma_r.A$ partitioning the sphere. This should be compared with the classical paradox of Hausdorff~\cite{Hausdorff14ma} who produced such a set $A$ that works,  apart from a countable subset of $\I S^2$ of errors, for every $r\ge 2$. (Note that we cannot take $r=2$ in Mycielski's result because $\I S^2$ is not 2-divisible.)

Let $\mu$ be the \emph{spherical measure} on $\I S^{d-1}$, which can be defined as the $(d-1)$-dimensional Hausdorff measure with respect to the standard \emph{arc-length 
distance} on the sphere (where the distance between $\B x,\B y\in \I S^{d-1}$ is the angle between the vectors $\B x$ and~$\B y$). We call a subset of $\I S^{d-1}$ \emph{measurable} if it belongs to the $\mu$-completion of the Borel $\sigma$-algebra.
Note that the paradoxical set $A$ in the results of Hausdorff~\cite{Hausdorff14ma} and Mycielski~\cite{Mycielski55} cannot be measurable with respect to the (rotation-invariant) measure $\mu$ on~$\I S^2$, for otherwise the existence of a partition $\gamma_1.A,\ldots,\gamma_r.A$ of $\I S^{d-1}$ up to a countable (and thus $\mu$-null) set implies that $\mu(A)=1/r$, a contradiction to $r$ assuming different values.
Mycielski~\cite{Mycielski57,Mycielski57a} asked if one can show that $\I S^2$ is  $r$-divisible without using the Axiom of Choice. Wagon~\cite[Question 4.15]{Wagon:btp} (or Question 5.15 in \cite{TomkowiczWagon:btp}) asked if the $3$-divisibility of $\I S^2$ can be shown with measurable sets (thus 
the Axiom of Choice can be applied on a $\mu$-null set). Measurable divisibility for higher dimensional spheres is easier because of a constructive way of lifting up a division from $\I S^{d-1}$ to~$\I S^{d+1}$.
It is known that $\Sd$ is $r$-divisible with measurable pieces for $r\ge 3$ and odd $d\ge 5$ (which follows from the proof of Theorem~6.6(b) in~\cite{TomkowiczWagon:btp}, see Lemma~\ref{lm:extend} here) and with Borel pieces for $r\ge 2$ and even $d\ge 2$ (see e.g.\ \TW{Theorem 6.6(a)}).

The above questions by Mycielski and Wagon are still open, although some related progress was obtained by Conley, Marks and Unger~\cite{ConleyMarksUnger20} whose general results imply that, unless $r=2$ and $d$ is odd, the sphere $\I S^{d-1}$ is $r$-divisible so that each piece has the \emph{property of Baire} (that is, under one of equivalent definitions, each piece can be represented as the symmetric difference of a Borel set and a meager set; for more details see e.g.\ the textbook on descriptive set theory by Kechris~\cite[Section~8.F]{Kechris:cdst}). The derivation of this result is given in Proposition~\ref{pr:Baire} here.

Here we propose to study the more general question of describing the set of those $r$-tuples $\B\gamma\in\SO(d)^r$ such that $\Sd$ is $\B\gamma$-divisible with measurable pieces.

First, we consider the case when the rotations are ``generic''. 
More precisely, 
%let $\QQ$ be the algebraic closure of the field of rational numbers and 
let us call an $r$-tuple of matrices $\B\gamma=(\gamma_1,\ldots,\gamma_r)\in \SO(d)^r$ \emph{generic} if, for every polynomial $p$ with rational coefficients in $d^2r$ variables, $p(\B\gamma)=0$ implies that $p(\B\beta)=0$ for every $\B\beta\in \SO(d)^r$, where e.g.\ $p(\B\gamma)$ denotes the value of $p$ on the $d^2r$ individual entries of the matrices corresponding to $\gamma_1,\ldots,\gamma_r$ under the standard basis of~$\I R^{d}$. In other words, this property states that if a polynomial with rational (equivalently, integer) coefficients vanishes on (the matrix entries of) $\B\gamma$ then it necessarily vanishes everywhere on~$\SO(d)^r$.

Our main result shows that \textbf{no} generic $\B\gamma$  works in the   measurable setting, even in a rather relaxed fractional version. 
%Let $\mu$ be the spherical measure on $\Sd$.

\begin{theorem}\label{th:main} Let $d\ge 2$ and $r\ge 2$ be integers. Let $(\gamma_1,\ldots,\gamma_r)\in \SO(d)^r$ be generic. Then every $f\in L^2(\I S^{d-1},\mu)$ with $\sum_{i=1}^r \gamma_i.f=1$ $\mu$-almost everywhere\ is the constant function $1/r$ $\mu$-almost everywhere, where $\gamma_i.f$ denotes the function that maps $\B x\in\I S^{d-1}$ to~$f(\gamma_i^{-1}.\B x)$.
\end{theorem}

 In sharp contrast, we can derive  with some extra work  from the results 
 %of Conley, Marks and Unger~
 in \cite{ConleyMarksUnger20} that  \textbf{every} generic $\B\gamma$ works with pieces that have the property of Baire.

\begin{proposition}
%[Conley, Marks and Unger~\cite{ConleyMarksUnger20}] 
\label{pr:Baire}
Let $r\ge 2$ and $d\ge 2$ be arbitrary integers, except if $d$ is odd then we require that $r\ge 3$. Let $(\gamma_1,\ldots,\gamma_r)\in\SO(d)^r$ be generic. Then
there is a subset $A$ of $\Sd$ with the property of Baire such that $\gamma_1.A,\ldots,\gamma_r.A$ partition $\Sd$.
\end{proposition}

	Theorem~\ref{th:main} and  Proposition~\ref{pr:Baire} add to a growing body of results in measurable combinatorics (see e.g.\ the recent survey by Kechris and Marks~\cite{KechrisMarks20survey}), where the requirements that the pieces are measurable and have the property of Baire respectively lead to different answers.

The following lemma 
%justifies the term ``generic''.
shows that, in various meanings, ``most'' elements of $\SO(d)^r$ are generic.

\begin{lemma}\label{lm:GNew} Let $r\ge 1$, $d\ge 2$ and $\Ng$ be the set of $r$-tuples in~$\SO(d)^r$ that are not generic. Then the following statements hold.
	
	\begin{enumerate}[(i)]
%	\item\label{it:GAlg}	For every $r$-tuple of matrices $\B\gamma$ in $\Ng$, the ${d\choose 2}r$-tuple of their entries strictly above the diagonals is algebraically dependent over~$\I Q$.
%		\item\label{it:GAlg} Every $r$-tuple of matrices $\B\gamma\in \SO(d)^r$ with their $r{n\choose 2}$ entries strictly above diagonal being algebraically independent over $\I Q$ is not in $\Ng$ (that is, is not generic).
		\item\label{it:Gmeas} The set $\Ng$ has measure 0 with respect to the Haar measure on the group $\SO(d)^r$.

		\item\label{it:GTop} The set $\Ng$ is a meager subset of $\SO(d)^r$ with respect to the topology induced by the Euclidean topology on $\I R^{d^2r}\supseteq \SO(d)^r$.
	\end{enumerate}
\end{lemma}

Also, by using some algebraic geometry, we can give a more concrete characterisation of generic $r$-tuples of rotations. In particular, the following lemma allows us to write an ``explicit'' generic point: just let the entries above the diagonals be sufficiently small reals that are algebraically independent over $\I Q$ and extend this to an element of $\SO(d)^r$ by Claim~\ref{cl:e} here.

\begin{lemma}
\label{lm:GPChar} 
Let $r\ge 1$, $d\ge 2$,  and $\B\gamma\in\SO(d)^r$. Then $\B\gamma$ is generic if and only if  the ${d\choose 2}r$-tuple of the matrix entries of $\B\gamma$ strictly above the diagonals is algebraically independent over~$\I Q$.\end{lemma}

In the extreme opposite case, we show that, for odd $d\ge 3$, $\B\gamma$-divisibility cannot be attained when $\B\gamma$ generates a finite subgroup of $\SO(d)$.

\begin{proposition}\label{pr:Euler} Let $d\ge 3$ be odd. Suppose that $\gamma_1,\ldots,\gamma_r\in\SO(d)$, $r\ge 3$, generate a finite subgroup $\Gamma\subseteq\SO(d)$. Then $\Sd$ is not $(\gamma_1,\ldots,\gamma_r)$-divisible.\end{proposition}

Some standard general results of Borel combinatorics (e.g.\ 
%this follows from 
Lemma~5.12 and Theorem~5.23 from~\cite{Pikhurko21bcc}) imply that if $\I S^{d-1}$ is $\B\gamma$-divisible and every orbit of the subgroup of $\SO(d)$ generated by $\gamma_1,\ldots,\gamma_r$ is finite, then there is a Borel $\B\gamma$-division. The following result gives that just one finite orbit is enough to convert a $\B\gamma$-division into a  measurable one.

\begin{proposition}
	\label{pr:FinOrbit} Let $d\ge 2$ and
	$\B\gamma=(\gamma_1,\ldots,\gamma_r)\in \SO(d)^r$. Let $\Gamma$ be the subgroup of
	$\SO(d)$ generated by $\gamma_1,\ldots,\gamma_r$. Suppose that there is
	$\B z\in\I S^{d-1}$ such that its $\Gamma$-orbit $\Gamma.\B z$ is finite. Then $\I S^{d-1}$ is $\B\gamma$-divisible if and only if $\I S^{d-1}$ is $\B\gamma$-divisible with
	measurable pieces.\end{proposition}

Of course, this leaves a wide range of unresolved cases. As an initial partial step, we completely characterise those $r$-tuples of rotations for which the circle $\I S^1$ is divisible with measurable pieces for~$r\le 3$.
% and, in some sense, also for $r=4$ (via Claim~\ref{cl:4m}).\medskip
% (in Section~\ref{se:d=2}). %Unfortunately, already the case $r=5$ seems to be messy.

%Variations of divisibility: dividing into pieces that are topologically homeomorphic to each other (Taniyama~\cite{Taniyama02geo})

\medskip
This paper is organised as follows. In Section~\ref{se:Harmonics} we give a quick overview of basic definitions and facts about spherical harmonics and use these to prove Theorem~\ref{th:main}, which is the main result of this paper. Proposition~\ref{pr:Euler} is proved in Section~\ref{se:Euler} using Euler's characteristic.  
Propositions~\ref{pr:FinOrbit} and~\ref{pr:Baire} are proved in Sections~\ref{se:FinOrbit} and \ref{se:Baire} respectively. In Section~\ref{se:HigherDim}  we describe the standard construction of how an $r$-division of $\Sd$ can be lifted to $\I S^{d+1}$ and observe that this gives   measurable pieces (Lemma~\ref{lm:extend}). 
In Section~\ref{se:d=2} we study  various versions of measurable divisibility when $d=2$; in particular, we characterise $r$-tuples $\B\gamma\in \SO(2)^r$ for which the circle $\I S^1$ is $\B\gamma$-divisible with measurable pieces for $r\le 3$. 
The rather technical Section~\ref{se:G} is dedicated to proving Lemmas~\ref{lm:GNew} and~\ref{lm:GPChar}. 
 Section~\ref{se:AlgGeo}
presents some basics of algebraic geometry.
% that we will need in the remaining part. 
In Section~\ref{se:SO} we prove some results about $\SO(d)^r$  and use them to prove Lemma~\ref{lm:GNew}. In particular, we show that the variety $\SO(d)^r\subseteq \I R^{d^2r}$ is irreducible and the entries above the diagonals form a transcendence basis for its function field. While these results are fairly standard, we present their proofs since we could not find any published statements that suffice for our purposes. 
%Using these results, we prove Lemma~\ref{lm:GNew}. 
In Section~\ref{se:aux} we prove an auxiliary lemma from algebraic geometry and use it to derive Lemma~\ref{lm:GPChar}.

\section{Spherical harmonics}\label{se:Harmonics}

Let an integer $d\ge 2$ be fixed throughout this section.

For an introduction to spherical harmonics on $\I S^{d-1}$ we refer to the book by Groemer~\cite{Groemer96gafssh} whose notation we generally follow. Recall that $\mu$ denotes the spherical measure on~$\I S^{d-1}$. Thus the total measure of the sphere is 
 $$
 \sigma_d:=\mu(\I S^{d-1})=\frac{2\pi^{d/2}}{\Gamma(d/2)}.
 $$
  As $d$ is fixed, the dependence on $d$ is usually not mentioned except for $\sigma_d$ (since $\sigma_{d-1}$ will also appear in some formulas). Also, the shorthand \emph{a.e.}\ stands for $\mu$-almost everywhere.

 By \gr{Lemma~1.3.1}, the density of the push-forward of $\mu$ under the projection to any coordinate axis is
\beq{eq:rho}
\rho(t):=\left\{\begin{array}{ll}
 \sigma_{d-1}\,(1-t^2)^{(d-3)/2}, & -1<t<1,\\
 0,& \mbox{otherwise.}
 \end{array}\right.
 \eeq
 
A polynomial $p\in\I R[\B x]$, $\B x=(x_1,\ldots,x_d)$, is called \emph{harmonic} if $\Laplace p=0$, where $$\Laplace:=\frac{\partial^2}{\partial x_1^2}+\ldots+\frac{\partial^2}{\partial x_d^2}$$ is the \emph{Laplace operator}. A \emph{spherical harmonic} is a function from $\Sd$ to the reals which is the restriction to $\I S^{d-1}$ of a harmonic polynomial on~$\I R^d$. Let $\C H$ be the vector space of all spherical harmonics.
For an integer $n\ge 0$, let $\C H_n\subseteq \C H$ be the linear subspace consisting of all functions $f:\I S^{d-1}\to \I R$ that are the restrictions to $\I S^{d-1}$ of some harmonic polynomial $p$ which is homogeneous of degree~$n$, where we regard the zero polynomial as homogeneous of any degree. By \gr{Lemma~3.1.3}, the polynomial $p$ is uniquely determined by $f\in\C H_n$, so we may switch between these two representations without mention. It can be derived from this (\gr{Theorem 3.1.4}) that the dimension of $\C H_n$ is 
\beq{eq:Nn}
 N_n:={d+n-1\choose n}-{d+n-3\choose n-2},
 \eeq
 where we agree that ${d+n-3\choose n-2}=0$ for $n=0$ or~$1$.
 
Let $\langle \cdot,\cdot\rangle$ denote the scalar product on $L^2(\I S^{d-1},\mu)$ (while $\B x\cdot\B y:=\sum_{i=1}^d x_iy_i$ denotes the scalar product of $\B x,\B y\in\I R^d$). 
It is known (\gr{Theorem 3.2.1}) that
\beq{eq:ortho}
\langle f,g\rangle =0,\quad\mbox{for all $f\in\C H_i$ and $g\in \C H_j$ with $i\not=j$},
\eeq
 that is,
 $\C H_0,\C H_1,\ldots$ are pairwise orthogonal subspaces
of $\C H\subseteq L^2(\I S^{d-1},\mu)$. 
%Thus we have that
% \beq{eq:Oplus} \C H=\oplus_{n=1}^\infty\C H_n.\eeq
 Note that the group $\SO(d)$ acts naturally on $L^2(\I S^{d-1},\mu)$ via the \emph{shift action} 
\beq{eq:Action}
 (\gamma.f)(\B v):=f(\gamma^{-1}.\B v),\quad \mbox{for }\gamma\in \SO(d),\ f\in L^2(\I S^{d-1},\mu),\ \B v\in\I S^{d-1}.
 \eeq
 Each space $\C H_n$ is invariant under this action (\gr{Proposition 3.2.4}) since, on $\I R^d$, 
 rotations preserve both the Laplace operator as well as the set of homogeneous degree-$n$ polynomials.

An important role is played by the \emph{Gegenbauer polynomials} $(P_0,P_1,\ldots)$ which are obtained from $(1,t,t^2,\ldots)$ by the Gram-Schmidt orthonormalization process on $L^2([-1,1],\rho(t)\dd t)$, except they are normalised to assume value 1 at $t=1$ (instead of being unit vectors in the $L^2$-norm). In the special case $d=3$ (when $\rho$ is the constant function), we get the \emph{Legendre polynomials}.
Of course, the degree of $P_n$ is exactly~$n$. Let us collect some of their standard properties that we will use.

\begin{lemma}\label{lm:P} For every integer $n\ge 0$ the following holds.
\begin{enumerate}[(i),nosep]
 \item\label{it:PQ} The polynomial $P_n$ has rational coefficients.
 \item\label{it:PHarm} For every $\B v\in\I S^{d-1}$, the function $P_n^{\B v}:\I S^{n-1}\to \I R$, defined by 
 \beq{eq:Pnv}
  P_n^{\B v}(\B x):=P_n(\B v\cdot \B x),\quad \mbox{for }\B x\in\I S^{d-1},
  \eeq
   belongs to $\C H_n$.
   \item\label{it:PBasis} There is a choice of $\B v_1,\ldots,\B v_{N_n}\in\I S^{d-1}$ such that the functions $P_n^{\B v_i}$, $i\in [N_n]$, form a basis of the vector space $\C H_n$.
   \item\label{it:PScalar} For every $\B u,\B v\in\I S^{d-1}$, we have $\langle P_n^{\B u},P_n^{\B v}\rangle =\frac{\sigma_d}{N_n} P_n(\B u\cdot \B v)$.
 \end{enumerate}
\end{lemma}

\begin{proof} Part~\ref{it:PQ} follows from the formula of Rodrigues (\gr{Proposition 3.3.7}) that provides an explicit expression for $P_n$, or from the standard recurrence relation that writes $P_{n+1}$ in terms of $P_{n}$ and $P_{n-1}$ for $n\ge 0$ (\gr{Proposition 3.3.11}) together with the initial values $P_{-1}(t):=0$ and~$P_0(t)=1$.

Part~\ref{it:PHarm}, namely the claim that each $P_n^{\B v}$ is in $\C H_n$, is one of the statements of \gr{Theorem 3.3.3}. 

Part~\ref{it:PBasis} is the content of~\gr{Theorem 3.3.14}. Alternatively, notice that under the action in~\eqref{eq:Action} we have for every $\B u,\B v\in\I S^{d-1}$  and $\gamma\in\SO(d)$ that $(\gamma.P_n^{\B v})(\B u)=P_n(\B v\cdot (\gamma^{-1}.\B u))=P_n((\gamma.\B v)\cdot \B u)$, that is, 
\beq{eq:gammaP}
\gamma.P_n^{\B v}=P_n^{\gamma.\B v}.
\eeq
 Thus the linear span of $P_n^{\B v}$, $\B v\in\I S^{d-1}$, is a non-zero $\SO(d)$-invariant subspace of $\C H_n$. By \gr{Theorem 3.3.4}, the only such subspace is $\C H_n$ itself, giving the required. 

Part~\ref{it:PScalar} follows from 
$$
\langle P_n^{\B u},P_n^{\B v}\rangle=\left(\int_{-1}^1 (P_n(t))^2\rho(t)\dd t\right) P_n(\B u\cdot \B v)=\frac{\sigma_d}{N_n}P_n(\B u\cdot \B v),
$$
where the first equality is a special case of the Funk-Hecke Formula (\gr{Theorem 3.4.1}) and the second equality (which by~\eqref{eq:rho} amounts to computing  the $L^2$-norm of any $P_n^{\B u}\in L^2(\I S^{d-1},\mu)$) is proved in \gr{Proposition 3.3.6}.\end{proof}

We need the following strengthening of Lemma~\ref{lm:P}.\ref{it:PBasis}, where we additionally require that the vectors $\B v_i$ are rational.  
%Recall that $\QQ$ denotes the algebraic closure of the field of rational numbers.

\begin{lemma}\label{lm:PBasis} For every integer $n\ge 0$, there is a choice of $\B v_1,\ldots,\B v_{N_n}\in\I S^{d-1}\cap\I Q^d$ such that the functions $P_n^{\B v_i}$, $i\in [N_n]$, form a basis of the vector space $\C H_n$.
\end{lemma}
\begin{proof} We pick $\B v_i$ in $\I S^{d-1}\cap\I Q^d$ one by one as long as possible so that the corresponding functions $P_n^{\B v_i}$ are linearly independent as elements of $\C H_n$. 
Let this procedure produce $\B v_1,\ldots,\B v_\ell$. Suppose that $\ell<{N_n}$ as otherwise
we are done. Let $\B v_{\ell+1}=\B x$, with $\V x=(x_1,\ldots,x_d)\in\I S^{d-1}$ being viewed as a vector of unknown variables. Consider the $(\ell+1)\times (\ell+1)$ matrix $M=M(\B x)$ with entries
\beq{eq:M}
M_{ij}:=\frac1{\sigma_{d}} \langle P_n^{\B v_i},P_n^{\B v_j}\rangle,\quad \mbox{for }i,j\in [\ell+1].
\eeq
In other words,  $\sigma_{d} M$ is the Gram matrix
of the vectors $P_n^{\B v_1},\ldots,P_n^{\B v_{\ell+1}}\in L^2(\I S^{d-1},\mu)$. In particular, the determinant $\det(M)$  of $M$ is 0 if and only if
$P_n^{\B v_{\ell+1}}$ is in the span of the (linearly independent) vectors~$P_n^{\B v_1},\ldots,P_n^{\B v_\ell}$ (by e.g.\ \cite[Theorem~7.2.10]{HornJohnson13ma}).

By Lemma~\ref{lm:P}.\ref{it:PScalar} we have that $M_{ij}=\frac1{N_n}P_n(\B v_i\cdot \B v_j)$. Thus the determinant of $M$ is a polynomial function of~$\B x$.
%Thus by Lemma~\ref{lm:P}.\ref{it:PQ}, the determinant of $M$ is a polynomial function in $\B x$ with all coefficients in $\I Q$. 
%By Items~\ref{it:PQ} and \ref{it:PScalar} of Lemma~\ref{lm:P}, the determinant $\det M$,  of $M$, is a polynomial function in $\B x$ with all coefficients in $\I Q$. 

By Lemma~\ref{lm:P}.\ref{it:PBasis} and $\ell<N_d$ (and the linear independence of $P_n^{\B v_1},\ldots,P_n^{\B v_\ell}$), there is some choice of $\B v_{\ell+1}\in\I S^{d-1}$ with $\det(M)\not=0$. 
That is, the polynomial $\det(M)$ is not identically zero on~$\I S^{d-1}$.

We  need the following easy claim that can be proved, for example, by induction on $d\ge 2$ with the base case $d=2$ following from $\I S^1$ containing all points of the form $\frac{1}{m^2+n^2}(m^2-n^2,2mn)$ for $(m,n)\in\I Z^2\setminus\{\,(0,0)\,\}$.

\begin{claim}\label{ls:DenseInS} For every $d\ge 1$, the set $\I S^{d-1}\cap\I Q^d$ of the points on the sphere with all coordinates rational is dense in $\I S^{d-1}$ with respect to the standard topology on the sphere (i.e.\ the one inherited from the Euclidean space $\I R^d\supseteq \I S^{d-1}$).\ecpf
	\end{claim}

\hide{
\bcpf We use induction on~$d$. If $d=1$, then $\I S^0=\{-1,1\}$ and the claim is trivially true. The validity of the case $d=2$ follows from $\I S^1$ containing all points of the form $\frac{1}{m^2+n^2}(m^2-n^2,\pm 2mn)$ for integers $m,n>0$. Indeed, in order to approximate an arbitrary $(x,y)\in\I S^1$ with e.g.\ $x,y\ge 0$ we just need that $\alpha:=n/m$ approximates well
$(\frac{1-x}{1+x})^{1/2}$ as then $\frac1{1+\alpha^2}(1-\alpha^2,2\alpha)\in \I S^1$ is close to $(x,y)$.

Let $d\ge 3$ and take an arbitrary point $\B x=(x_1,\ldots,x_d)$ of~$\I S^{d-1}$.
% Since $X$ is invariant under changing signs of individual coordinates, assume that $x_i\ge 0$ for each $i\ge 1$. 
Assume that,  for each $i\in [d]$, we have $x_i\not=0$ for otherwise we can approximate $\B x$ within the $(d-2)$-dimensional unit sphere $\I S^{d-1}\cap \{\B z\in \I R^d\mid z_i=0\}$ by induction. To approximate $\B x$, take any $(y_1,y_2)\in \I S^2\cap \I Q^2$ with $y_1^2/y_2^2$ approximating well $\alpha/\beta$ where $\alpha:=\sum_{i=1}^{d-2} x_i^2>0$ and $\beta:=x_{d-1}^2+x_{d}^2>0$ and let 
$$
\B z:=y_1 (\B v_1,0,0)+y_2(0,\ldots,0,\B v_2)\in \I S^{d-1}\cap \I Q^d,
$$ where $\B v_1\in\I S^{d-3}$ and $\B v_2\in\I S^2$ are rational vectors that approximate well the unit vectors $\alpha^{-1/2}\,(x_1,\ldots,x_{d-2})$ and $\beta^{-1/2}\, (x_{d-1},x_d)$ respectively.\ecpf
}

Since $\det(M)$, as a polynomial function  of $\V x\in \I S^{d-1}$,
is continuous and not identically zero, it has to be non-zero on some point $\B x$ of the dense subset~$\I S^{d-1}\cap\I Q^d$. Thus, if we let $\B v_{\ell+1}$ to be such a vector $\B x$, then the functions
$P_n^{\B v_1},\ldots,P_n^{\B v_{\ell+1}}\in L^2(\I S^{d-1},\mu)$ are linearly independent. This contradiction to the maximality of $\B v_1,\ldots,\B v_\ell$ proves the lemma.\end{proof}

For an integer $n\ge 0$, an $r$-tuple $\B\gamma=(\gamma_1,\ldots,\gamma_r)\in\SO(d)^r$ and a unit vector $\B v\in\I S^{d-1}$ define
\beq{eq:Gv}
G^{\B v}_{n,\B\gamma}:= \sum_{i=1}^r P_n^{\gamma_i^{-1}.\B v}.
\eeq
 By Lemma~\ref{lm:P}.\ref{it:PHarm}, each function $G^{\B v}_{n,\B\gamma}:\I S^{d-1}\to\I R$, as a linear combination of some spherical harmonics $P_n^{\gamma_i^{-1}.\B v}\in\C H_n$, is itself in $\C H_n$. 

\begin{lemma}\label{eq:GenericGammas} If $\B\gamma\in\SO(d)^r$ is generic then, for every integer $n\ge 0$, the linear span of $\{G^{\B v}_{n,\B\gamma}\mid \B v\in\I S^{d-1}\}$ is the whole space $\C H_n$.
 \end{lemma}
\begin{proof} By Lemma~\ref{lm:PBasis}, we can fix some vectors $\B v_1,\ldots,\B v_{N_n}\in\I S^{d-1}\cap\I Q^d$ such that 
$P_n^{\B v_1},\ldots,P_n^{\B v_{N_n}}$ form a basis for~$\C H_n$. Let $\B\beta=(\beta_1,\ldots,\beta_r)$ be an arbitrary element of $\SO(d)^r$ (not necessarily generic). 
 Consider the $N_n\times N_n$ matrix $L=L(\B\beta)$ with entries
 $$
 L_{ij}:=\frac1{\sigma_{d}} \langle G^{\B v_i}_{n,\B\beta},P_n^{\B v_j}\rangle,\quad\mbox{for } i,j\in [N_n].
 $$

Recall that the vectors $P_n^{\B v_i}$, $i\in [N_n]$, form a (not necessarily orthonormal) basis of the linear space $\C H_n$. Write the vectors $G^{\B v_i}_{n,\B\beta}$ in this basis: 
$$
 (G^{\B v_1}_{n,\B\beta},\ldots,G^{\B v_{N_n}}_{n,\B\beta})^T=A\,(P_n^{\B v_1},\ldots,P_n^{\B v_{N_n}})^T,
 $$
 for some $N_n\times N_n$ matrix~$A$. Then $L$ is the matrix product $AM$, where $M$ is the Gram matrix of the vectors $P_n^{\B v_i}$ multiplied by the constant $\sigma_d^{-1}$ (that is, the entries of $M$ are defined by the formula in~\eqref{eq:M}). The matrix $M$ is non-singular by the linear independence of $P_n^{\B v_i}$, $i\in [N_n]$. 
 Thus
$\det (L)\not=0$ if and only if  $G^{\B v_1}_{n,\B\beta},\ldots,G^{\B v_{N_n}}_{n,\B\beta}$ are linearly independent as vectors in~$\C H_n$.

 %View the $rn^2$ entries of $\gamma_1,\ldots,\gamma_r$ as ``unknowns''. 
 By Lemma~\ref{lm:P}.\ref{it:PScalar}, we have for every $i,j\in [N_d]$ that
 $$
  L_{ij}:=\frac1{\sigma_d} \sum_{s=1}^r \langle P_n^{\beta_s^{-1}.\B v_i},P_n^{\B v_j}\rangle 
  =\frac{1}{N_n}\sum_{s=1}^r P_n((\beta_s^{-1}.\B v_i)\cdot \B v_j)=\frac{1}{N_n}\sum_{s=1}^r P_n(\B v_i\cdot (\beta_s.\B v_j)).
  $$

Since $\B v_1,\ldots,\B v_{N_n}$ are fixed, this writes each $L_{ij}$ as a polynomial in the $d^2r$ entries of the matrices $\beta_1,\ldots,\beta_r$. Moreover, all coefficients of this polynomial are rational  since each $\B v_i$ belongs to $\I Q ^d$ and all coefficients of $P_n$ are rational
by Lemma~\ref{lm:P}.\ref{it:PQ}. Thus the determinant of $L$ is equal to $p(\B\beta)$ for some polynomial $p$ with coefficients in~$\I Q $.
% that depend only on~$\B v_1,\ldots,\B v_{N_n}$. 
%So far we have not have used the fact that $\B\gamma$ is generic (only that each $\gamma_i$ is in $\SO(d)$).

Note that if we let each $\beta_i$ be the identity matrix $I_d$, then $G^{\B v}_{n,\B\beta}$ becomes $rP_n^{\B v}$ for every $\B v\in \I S^{d-1}$ and we have  $L_{ij}=\frac r{\sigma_{d}} \langle P_n^{\B v_i},P_n^{\B v_j}\rangle$ for $i,j\in[N_n]$ and $\det (L)\not=0$ (since $P_n^{\B v_1},\ldots,P_n^{\B v_{N_n}}$ are linearly independent). Thus $p(I_d,\ldots,I_d)\not=0$. Since $\B\gamma\in\SO(d)^r$ is generic, we have that $p(\B\gamma)\not=0$, that is, the matrix $L$ for $\B\beta:=\B\gamma$ is non-singular. This means that the functions $G^{\B v_i}_{n,\B\gamma}$, $i\in [N_n]$,
%, from~\eqref{eq:Gv} 
are linearly independent. Since they all lie in $\C H_n$ and their number equals the dimension of this linear space, they span~$\C H_n$. The lemma is proved.\end{proof}
 
Given the above auxiliary results, we can derive Theorem~\ref{th:main} rather easily.

\begin{proof}[Proof of Theorem~\ref{th:main}.] Recall that $\B \gamma=(\gamma_1,\ldots,\gamma_r)\in\SO(d)^r$, $r\ge 2$, is generic  and we have to show that $\I S^{d-1}$ is not ``fractionally'' $\B \gamma$-divisible.

So take any $f\in L^2(\I S^{d-1},\mu)$ such that $\sum_{i=1}^r \gamma_i.f=1$ a.e. Since spherical harmonics are dense in $L^2(\I S^{d-1},\mu)$ (\gr{Corollary~3.2.7}) and we have the direct sum $\C H =\oplus_{n=0}^\infty \C H_n$ whose components are orthogonal to each other by~\eqref{eq:ortho}, we can uniquely write $f=\sum_{n=0}^\infty F_n$  in $L^2(\I S^{d-1},\mu)$ with $F_n\in\C H_n$ for every~$n\ge 0$. Since the action of $\SO(d)$ preserves each space $\C H_n$ as well as the
scalar product on $L^2(\I S^{d-1},\mu)$,  we have that
$\gamma.f=\sum_{n=0}^\infty \gamma.F_n$ is the harmonic expansion of $\gamma.f\in L^2(\I S^{d-1},\mu)$.

Take any integer $n\ge 1$. Recall that the sum  $\sum_{i=1}^r \gamma_i.f$ is a constant function $1$ a.e.
 By~\eqref{eq:ortho}, the invariance of the scalar product under $\SO(d)$ and by~\eqref{eq:gammaP}, we have that, for every~$\B v\in\I S^{d-1}$, 
 \begin{eqnarray*}
 0&=&\langle P_n^{\B v},1\rangle\ =\ \langle P_n^{\B v},\gamma_1.f+\ldots+\gamma_r.f\rangle\ =\ \langle P_n^{\B v},\gamma_1.F_n+\ldots+\gamma_r.F_n\rangle\\
 &=&\langle \gamma_1^{-1}.P_n^{\B v}+\ldots+\gamma_r^{-1}.P_n^{\B v},F_n\rangle\ =\ \langle G^{\B v}_{n,\B\gamma}, F_n\rangle,
 \end{eqnarray*}
 where $G^{\B v}_{n,\B\gamma}$ was defined by~\eqref{eq:Gv}.
 Since the functions $G^{\B v}_{n,\B\gamma}$, $\B v\in\I S^{d-1}$, span the whole space $\C H_n$ by Lemma~\ref{eq:GenericGammas}, we must have that $F_n=0$.
 
 As $n\ge 1$ was arbitrary, we have that $f$ is a constant function a.e.\ (whose value must be~$1/r$). This finishes the proof of Theorem~\ref{th:main}.\end{proof}

\hide{
Note that 
$$
 \sigma_dF_0=\langle 1,1\rangle F_0=\langle f,1\rangle=\frac{\sigma_d}{r}
$$
}

\begin{remark}
The statement of Theorem~\ref{th:main} remains true also when $\gamma_r=I_d$ and $(\gamma_1,\ldots,\gamma_{r-1})$ is a generic point of $\SO(d)^{r-1}$. One way to see this is to run the same proof except the $r$-th component of each encountered $r$-tuple of matrices is always set to be the identity matrix~$I_d$.  
\end{remark}

\section{Rotations generating a finite subgroup}
\label{se:Euler}

\begin{proof}[Proof of Proposition~\ref{pr:Euler}.] We have to show that an even-dimensional sphere $\I S^{d-1}$ is not $(\gamma_1,\ldots,\gamma_r)$-divisible if the subgroup $\Gamma$ of $\SO(d)$ generated by the rotations $\gamma_1,\ldots,\gamma_r$ is finite.

Since $d$ is odd, the $2$-divisibility of $\I S^{d-1}$ is impossible because of a fixed point of~$\gamma_1^{-1}\gamma_2$.
So assume that $r\ge 3$. Let 
$$V:=\Gamma.\{\pm\B e_1,\ldots,\pm\B e_d\},
$$
 that is, we take all possible images of the standard basis vectors and their negations when moved by $\Gamma$. 
Clearly, the set $V$ is a finite. Let $P$ be the convex hull of $V$. Then $P$ is a full-dimensional polytope containing $\Zero$ in its interior (as already the convex hull of $\{\pm\B e_1,\ldots,\pm\B e_d\}\subseteq V$ has these properties). 
Its boundary $\partial P$ is homeomorphic to $\Sd$ by the map that sends $\B x\in\partial P$ to $\B x/\|\B x\|_2\in\Sd$. 

Let a \emph{hyperplane} mean a $(d-1)$-dimensional affine subspace of~$\I R^d$.
Identify each oriented hyperplane $H\subseteq \I R^d$ with the pair $(\B n,a)\in\I S^{d-1}\times\I R$ so that 
$$
 H=\{\B x\in\I R^d\mid \B n\cdot \B x=a\}.
$$ 
Its \emph{open half-spaces} are $H^+:=\{\B x\in\I R^d\mid \B n\cdot \B x>a\}$ and $H^-:=\{\B x\in\I R^d\mid \B n\cdot \B x<a\}$. 
Call $H$ \emph{supporting} if $H\cap P\not=\emptyset$ and $H^-\cap P=\emptyset$. Call $H$ a \emph{facet hyperplane} if it is supporting and $\affdim(H\cap P)=d-1$, where $\affdim(X)$ denotes the dimension of the affine subspace of $\I R^d$ spanned by~$X$. 
%(Since the origin is not in the interior of $H$, the last

The intersections of supporting hyperplanes with $\partial P$ represent the boundary of the polytope $P$ as a CW-complex. Namely, for $i\in \{0,\ldots,d-1\}$, its $i$-dimensional cells are precisely the \emph{$i$-dimensional faces} of $P$, that is, the convex hulls of the sets in 
$$
\C C_{i}:=\{X\subseteq V\mid  \affdim(X)=i\ \&\ \exists \mbox{ supporting hyperplane $H$ with $H\cap V=X$}\}.
$$

For a finite non-empty set $X\subseteq \I R^d$, let 
$
\B m_X:=\frac1{|X|} \sum_{\B x\in X} \B x
$ 
be the \emph{centre of mass} of~$X$.

Let us show that for every $i\in \{0,\ldots,d-1\}$ and  distinct $X,Y\in\C C_{i}$ we have $\B m_X\not=\B m_Y$. 
As it is well-known,  see e.g.~\cite[Theorem 3.1.7]{Grunbaum03cp}, we can pick facet hyperplanes $H_1,\ldots,H_k$ such that $V\cap (\cap_{j=1}^{k} H_j)=X$. 
Since $X\not=Y$, the affine subspaces that these two sets span differ. Since these subspaces have the same dimension, there is $\B y\in Y$ not in the affine span of~$X$.
Since $\B y\in V$ and each $H_j$ is supporting, there is $j\in[k]$ such that $\B y$ belongs to the open half-space $H_j^+$. From $Y\subseteq H_j\cup H_j^+$, it follows that $\B m_Y$ belongs to $H_j^+$ and cannot be equal to $\B m_X\in H_j$, as claimed.

Also, it holds that $\B m_X\not=\Zero$ for any $X\in\C C_{i}$. Indeed, with $H_1,\ldots,H_k$ as above we have that $\Zero$, which is in the interior of $P$, belongs to, say, the open half-space $H_1^+$ so cannot be equal to $\B m_X\in
%\cap_{j=1}^k H_j\subseteq 
H_1$.

Thus $|M_i|=|\C C_i|$, where  $M_i:=\{\B m_X/\|\B m_X\|_2\mid X\in \C C_i\}\subseteq \Sd$ denotes the set of the normalised centres of mass of the vertex sets of $i$-dimensional faces. Clearly, the set family $\C C_i$ is invariant under the natural action of $\Gamma$ on finite subsets of $\Sd$. Thus the set $M_i\subseteq \Sd$ is also $\Gamma$-invariant. 

Since $d$ is odd, the Euler characteristic $\chi(\I S^{d-1})$ of the $(d-1)$-dimensional sphere is $2$, see e.g.\ \cite[Remark 4.2.21]{Weintraub14fat}. Since the faces of $\partial P$ give a representation of the sphere as a CW-complex, we have (by e.g.\ \cite[Theorem 4.2.20]{Weintraub14fat}) that
$$
 2=\chi(\Sd)=\sum_{i=0}^{d-1} (-1)^i|\C C_i|,
 $$
 Thus, for at least one $i\in \{0,\ldots,d-1\}$, it holds that $r\ge 3$ does not divide $|\C C_i|=|M_i|$. By the $\Gamma$-invariance of $M_i$, there is no choice of $A\cap M_i$ such that its translates by $\gamma_1,\ldots,\gamma_r$ partition $M_i$. Thus $\I S^{d-1}$ is not $(\gamma_1,\ldots,\gamma_r)$-divisible.\end{proof}

\begin{remark} 
Under the assumptions of Proposition~\ref{pr:Euler}, its proof gives that if there are $d$ linearly independent vectors on $\I S^{d-1}$ such that each has a finite orbit under $\Gamma$ (where  some of these orbits may coincide) then $\I S^{d-1}$ is not $\B\gamma$-divisible. However, this seemingly weaker assumption is equivalent to the assumption that $\Gamma$ is finite (e.g.\ via a version of Claim~\ref{cl:Fin} below).
\end{remark}

\section{Actions with a finite orbit}\label{se:FinOrbit}

Here we prove Proposition~\ref{pr:FinOrbit} that, in the presence of at least one finite orbit, $\B\gamma$-divisibility is equivalent to measurable $\B\gamma$-divisibility.

\begin{proof}[Proof of Proposition~\ref{pr:FinOrbit}.] Recall that $\Gamma$ is the subgroup of $\SO(d)$ generated by~$\gamma_1,\ldots,\gamma_r$. For $\B x\in\I S^{d-1}$, let $L_{\B x}$ be the linear subspace of $\I R^d$ spanned by $\Gamma.\B x\subseteq \I R^d$.

\begin{claim}\label{cl:Inv}
	For every $\B x\in\I S^d$, both $L_{\B x}\subseteq\I R^d$ and its orthogonal complement $L_{\B x}^\perp\subseteq\I R^d$ are invariant under the action of $\Gamma$ on~$\I R^d$.
\end{claim}
 \bcpf Any $\gamma\in\Gamma$ permutes the set $\Gamma.\B x$. Since $\gamma$ is a linear map, it preserves the linear subspace $L_{\B x}$ spanned by~$\Gamma.\B x$. Thus $L_{\B x}$ is $\Gamma$-invariant.
 
Since $\Gamma$ consists of orthogonal matrices, its action preserves the scalar product on~$\I R^d$. Thus if $\B y\in\I R^d$ is orthogonal to $L_{\B x}$ then,  for every $\gamma\in\Gamma$, we have that $\gamma.\B y$ is orthogonal to $\gamma.L_{\B x}=L_{\B x}$. It follows that $L_{\B x}^\perp$ is $\Gamma$-invariant.\ecpf

Recall that $\B z\in\I S^{d-1}$ is a vector such that its orbit $\Gamma.\B z$ is finite. Let $z_1,\ldots,z_n$ be the elements of~$\Gamma.\B z$.

\begin{claim}\label{cl:Fin} 
	If $\B x\in L_{\B z}\cap \I S^{d-1}$ then $|\Gamma.\B x|\le n!$.
	% (in particular $\B x$ has a finite orbit under the action by~$\Gamma$).
\end{claim}

\bcpf 
%For each $i\in [n]$ pick some $\alpha_i\in \Gamma$ with $\alpha_i.\B z=\B z_i$.
Write $\B x\in L_{\B z}$ as $\sum_{i=1}^n c_i\B z_i$ for some reals $c_1,\ldots,c_n$. For every $\alpha\in\Gamma$, we have by linearity that $\alpha.\B x=\sum_{i=1}^n c_i (\alpha .\B z_i)$. Since $\B z_1,\ldots,\B z_n$ enumerate a whole orbit of $\Gamma$, the element $\alpha\in\Gamma$ permutes these vectors. Thus every element of $\Gamma.\B x$ is of the form $\sum_{i=1}^n c_{i} \B z_{\sigma(i)}$ for some permutation $\sigma$ of $[n]$. Thus $\Gamma.\B x$ indeed has at most $n!$ elements.\ecpf

Now we are ready to prove the (non-trivial) forward direction of Proposition~\ref{pr:FinOrbit}.  By rotating the sphere (and moving $\B z$ and conjugating $\gamma_i$'s accordingly), we can assume that $L_{\B z}=\I R^m\times \Zero$ and $L_{\B z}^\perp=\Zero\times \I R^{d-m}$ for some $m\in [d]$. 
%Enumerate $\Gamma.\B z=\{\B z_1,\ldots,\B z_n\}$. 
By Claim~\ref{cl:Inv}, every matrix $\gamma_i$, $i\in [r]$, consists now of two diagonal blocks that correspond to some $\alpha_i\in \mathrm{O}(m)$ and $\beta_i\in \mathrm{O}(d-m)$. (Note that these matrices may have determinant~$-1$.) When we write a vector in $\I R^d$ as $(\B x,\B y)$, we mean that $\B x\in\I R^m$ and $\B y\in\I R^{d-m}$; thus $\gamma_i.(\B x,\B y)=(\alpha_i.\B x,\beta_i.\B y)$.

Fix $C\subseteq \I S^{d-1}$ such that $\gamma_1.C,\ldots,\gamma_r.C$ partition~$\I S^{d-1}$. By the invariance of $L_{\B z}$ and $L_{\B z}^\perp$, the translates of the set $C\cap (\I R^{m}\times \Zero)$ (resp.\ $C\cap (\Zero\times \I R^{m-d})$) by $\gamma_1,\ldots,\gamma_r$ partition $\I S^{m-1}\times\Zero$ (resp.\ $\Zero\times \I S^{d-m-1}$). By Claim~\ref{cl:Fin}, every orbit of the 
action of $\Gamma$ on the invariant subset $X:=\I S^{m-1}\times\Zero$ has at most $n!$ elements. Obviously, the same holds for the action  on~$\I S^{m-1}$ of the subgroup $\Gamma'\subseteq \mathrm{O}(m)$ generated by $\alpha_1,\ldots,\alpha_r$. 
Fix a Borel total order on $\I S^{m-1}$ (e.g.\ the restriction of the lexicographic order on $\I R^m$) and let $A'\subseteq X$ be obtained by picking from every orbit $\Gamma'.\B x\subseteq \I S^{m-1}$ the lexicographically smallest subset such that its translates by $\alpha_1,\ldots,\alpha_r$ partition $\Gamma'.\B x$. Such a set always exists since $\{\B y\in\Gamma'.\B x\mid (\B y,\Zero)\in C\}$ is one possible choice. 
%The action $\Gamma\actson X$ is Borel (as the restriction of the continuous action $\Gamma\actson\I S^{d-1}$ to the closed set $X$) and, 
In the terminology of~\cite{Pikhurko21bcc}, the set $A'$ can be computed by a local rule of radius $n!$ on the \emph{coloured Schreier digraph} of $\Gamma'\actson \I S^{m-1}$ (where the vertex set is $\I S^{d-1}$ and we put a directed colour-$i$ arc from $\B y$ to $\alpha_i.\B y$
for all $\B y\in \I S^{m-1}$ and $i\in [r]$). As the action is Borel, this is known  to imply (see e.g.\ \cite[Lemma~5.17]{Pikhurko21bcc}) that the constructed set $A'\subseteq \I S^{m-1}$ is Borel.
Define
%Now we expand $A'$ in a Borel way so that we extend the $\B\gamma$-division from $X$ to $\I S^{d-1}\setminus L_{\B z}^\perp$. Namely, let 
\begin{eqnarray*}
 A&:=&\textstyle \bigcup_{\rho\in [0,1)} (\sqrt{1-\rho^2}\, A'\times \rho\,\I S^{d-m-1})\\
 &=&\textstyle \bigcup_{\rho\in [0,1)} \{ (\sqrt{1-\rho^2}\, \B x, \rho\, \B y)\mid \B x\in A',\ \B y\in \I S^{d-m-1}\}
\end{eqnarray*}
 and $B:=C\cap (\Zero\times \I R^{m-d})$.
Then  $\gamma_1.A,\ldots,\gamma_r.A$ partition $\I S^{d-1}\setminus (\Zero\times \I S^{d-m-1})$ and, as we observed earlier, $\gamma_1.B,\ldots,\gamma_r.B$ partition $\Zero\times \I S^{d-m-1}$. Thus $A\cup B$ 
witnesses the $\B\gamma$-divisibility of~$\I S^{d-1}$. Note that the set $B$, which lies inside the intersection of $\I S^{d-1}$ with the linear subspace $L_{\B z}^\perp$ of dimension less than $d$, has measure zero. On the other hand, the set $A$ can be equivalently defined as the pre-image of the Borel set $A'\times \I R^{d-m}$ under the natural homeomorphism between $\I S^{d-1}\setminus (\Zero\times \I S^{d-m-1})$ and $\I S^{m-1}\times \I R^{d-m}$ that maps $(\B x,\B y)$ to $(\B x/\|\B x\|_2,\B y/\|\B x\|_2)$. Thus $A$ is Borel and $A\cup B$ is measurable, proving the proposition.\end{proof}

\begin{remark}
One can show via Claims~\ref{cl:Inv} and~\ref{cl:Fin} that if $d=3$ and a subgroup $\Gamma\subseteq \SO(d)$ has a finite orbit of size at least $3$, then $\Gamma$ is finite (and thus Proposition~\ref{pr:Euler} applies). However, this implication is not true in general for $d\ge 4$. For example, we can take the subgroup of $\SO(d)$ generated by a diagonal block matrix $M$ whose first (resp.\ second) block is a $2\times 2$ special orthogonal matrix of order 3 (resp.\ of infinite order), while all remaining blocks are the $1\times 1$ identity matrices. Then $M$ has an infinite order (coming from the second block) but its action on $\I S^{d-1}$ has an orbit with exactly 3 elements (e.g.\ the orbit of the first standard basis vector $(1,0,\ldots,0)$).
\end{remark}

\section{Measurable divisibility of higher-dimensional spheres}
\label{se:HigherDim}
 
 As we mentioned in the Introduction, $\Sd$ is $r$-divisible with Borel pieces for every $r\ge 2$ and even $d\ge 2$
 (\TW{Theorem 6.6(a)}). 
 The proof of~\cite[Theorem~6.6(b)]{TomkowiczWagon:btp} for any $r\ge 3$ and odd $d\ge 5$ gives measurable pieces. Since this conclusion does not seem to be explicitly stated anywhere in \cite{TomkowiczWagon:btp}, we provide the simple proof from~\cite{TomkowiczWagon:btp}.
 
 \begin{lemma}\label{lm:extend} For any $d\ge 5$ and $r\ge 3$, $\I S^{d-1}$ is $r$-divisible  with   measurable pieces.
 \end{lemma}
 \begin{proof} Informally speaking, we will use the Borel $r$-divisibility of $\I S^1$ in the last two coordinates of $\I S^{d-1}\subseteq \I R^{d}$, resorting to the $r$-divisibility of $\I S^{d-3}$ only on the null set of points where the last two coordinates are zero.

 Namely, choose rotations $\alpha_1,\ldots,\alpha_r\in \SO(d-2)$ and a (not necessarily measurable) subset $A\subseteq \I S^{d-3}$ such that $\alpha_1.A,\ldots,\alpha_r.A$ partition $\I S^{d-3}$, which is 
 possible by e.g.~\cite[Theorem~6.6]{TomkowiczWagon:btp}.
 Let $\beta\in\SO(2)$ be the rotation of the circle $\I S^1$ by the angle~$2\pi/r$. (Thus the order of $\beta$, as an element of the group $\SO(2)$, is~$r$.) For $i\in[r]$, let $\gamma_i$ send $(\B x,\B y)\in \I R^{d-2}\times\I R^2$ to $(\alpha_i.\B x,\beta^i.\B y)$, where we view $\SO(m)$ as also acting on $\I R^m$.
 % for an integer $m\ge 1$. 
 Clearly, $\gamma_i$ preserves both the scalar product on $\I R^{d}$ and the orientation; thus it is an element of~$\SO(d)$. 
 %(In terms of matrices, $\gamma_i$ consists of two diagonal blocks, the $(d-2)\times (d-2)$ matrix of $\alpha_i$ and the $2\times 2$ matrix of the rotation~$\gamma$.) 
 
 \hide{
 	One way to complete the proof now is to invoke Proposition~\ref{pr:FinOrbit}, observing that the orbit of $(0,\ldots,0,1)\in\I S^{d-1}$ under the group generated by $\gamma_1,\ldots,\gamma_r$ is finite, namely has exactly $m$ elements. However, we present instead a direct and simple description of possible pieces.
 }
 
 Let $B:=\{(\cos \theta,\sin \theta)\mid 0\le \theta< 2\pi/r\}\subseteq\I S^1$. 
 Then the half-open arcs $\beta.B,\ldots,\beta^{r}.B$ partition $\I S^1$. Let  $C:=A'\cup B'$, where
 $A':=A\times\{\,(0,0)\,\}$ and
 $$
 \textstyle
 B':=\bigcup_{\rho\in [0,1)} (\rho\,\I S^{d-3}\times \sqrt{1-\rho^2}\, B)
 %=\cup_{\rho\in [0,1)} \{ (\rho\, \B x, \sqrt{1-\rho^2}\, \B y)\mid \B x\in \I S^{d-3},\ \B y\in B\}
 .
 $$ 
 Clearly, $A'$ is a $\mu$-null subset of~$\Sd$ and $B'$ is a Borel subset of $\Sd$. Thus $C$ is   measurable. Also, $\gamma_1.C,\ldots,\gamma_r.C$ partition $\I S^{d-1}$. Indeed, $\gamma_i.A'= \alpha_i.A\times\{\,(0,0)\,\}$, $i\in [r]$, partition $\I S^{d-3}\times\{\,(0,0)\,\}$ while $\gamma_i.B'=\cup_{\rho\in [0,1)}(\rho\,\I S^{d-3}\times \sqrt{1-\rho^2}\, (\beta^i.B))$, $i\in [r]$, partition the rest of $\I S^{d-1}$.\end{proof}

\section{Measurable divisibility for $d=2$ and $r\le 4$}
\label{se:d=2}

We parametrise $\I S^{1}=\{(\cos t,\sin t)\mid t\in [0,2\pi)\}$ and use the parameter $t$ instead of the Cartesian coordinates. Thus we have the interval $[0,2\pi)$ with $\mu$ being the Lebesgue measure on it. The space $\C H_n$ for $n\ge 1$ becomes the span of $\cos nt$ and $\sin nt$ (while,  of course,  $\C H_0$ consists of all constant functions). Here, the harmonic expansion is nothing else as the Fourier series. We identify $\SO(2)$ with the additive group $\I T:=\I R/2\pi\I Z$ of reals taken modulo $2\pi$. Thus the action of  $\gamma\in\I T$ on $[0,2\pi)$ is to send $t\in [0,2\pi)$ to $t+\gamma\pmod{2\pi}$. We also identity $[0,2\pi)$ with $\I T$; thus we have the natural action $\I T\actson\I T$.  

Let us investigate various possible versions of ``measurable'' divisibility, stated in terms of the action~$\I T\actson\I T$. Let $\Borel_r$ (resp.\ $\C M_r$)  consist of those $r$-tuples $(t_1,\ldots,t_{r})\in \I T^{r}$ for which there is a Borel (resp.\ measurable) subset $A\subseteq \I T$ such that $t_1+A,\ldots,t_r+A$ partition~$\I T$, where we denote $t+A:=\{t+a\mid a\in A\}$. Also, let $\C M'_r$  consist of those $(t_1,\ldots,t_{r})\in \I T^{r}$ for which
there is a measurable (equivalently, Borel) $A\subseteq \I T$ such that the translates $t_1+A,\ldots,t_r+A$ are pairwise disjoint and the set of elements of $\I T$ not covered by them has measure zero.
Finally, let $\Frac_r$ consist of those $(t_1,\ldots,t_{r})\in \I T^{r}$ for which
there is $f\in L^2([0,2\pi),\mu)$ such that $t_1.f+\ldots+t_r.f=1$ a.e.\ while $f\not=1/r$ on a set of positive measure. As it is easy to see, the definition of $\Frac_r$ does not change if we require  $t_1.f+\ldots+t_r.f=1$ to hold everywhere.  Trivially, it holds that 
$$
 \Borel_r\subseteq \C M_r\subseteq\C M'_r\subseteq \Frac_r.
$$

First, we investigate $\Frac_r$. Suppose that we have some $f\in L^2([0,2\pi),\mu)$ such that $t_1.f+\ldots+t_r.f=1$ a.e. Take the Fourier series,
$$
f(t)=c_0+\sum_{n=1}^\infty (c_n \cos nt+s_n\sin nt),\quad \mbox{for a.e.\ }t\in [0,2\pi).
$$
Clearly, $c_0=1/r$. For  $i\in [r]$, by translating everything by $t_i$ we get that
$$
(t_i.f)(t)=\frac1r+\sum_{n=1}^\infty (c_n \cos n(t-t_i)+s_n\sin n(t-t_i)),\quad \mbox{for a.e.\ }t\in [0,2\pi).
$$
Summing this up for all $i\in [r]$ and using the formula for the sine and the cosine of a difference of two angles, we get that for a.e.\ $t\in [0,2\pi)$
\begin{eqnarray*}
	1
	&=&1+ \sum_{i=1}^r \left(\sum_{n=1}^\infty c_n(\cos nt  \cos nt_i+\sin nt\sin nt_i)+s_n(\sin nt  \cos nt_i-\cos nt\sin nt_i)\right)\\
	&=&1+\sum_{n=1}^\infty \left(\sum_{i=1}^r(c_n\cos nt_i-s_n \sin nt_i)\cos nt+\sum_{i=1}^r(c_n\sin nt_i+s_n \cos nt_i)\sin nt\right).
\end{eqnarray*}
(Recall that $\sum_{i=1}^r t_i.f=1$ a.e.)

Let $n\ge 1$. By the uniqueness of the Fourier coefficients, we have that
\begin{eqnarray*}
	\sum_{i=1}^r(c_n\cos nt_i-s_n \sin nt_i)&=&0,\mbox{ and}\\
	\sum_{i=1}^r(c_n\sin nt_i+s_n \cos nt_i)&=&0.
\end{eqnarray*}
Suppose that $(c_n,s_n)\not=(0,0)$. If we multiply the above equations by $c_n$ and $s_n$ (resp.\ by $-s_n$ and $c_n$) and add up, we get after dividing by $c_n^2+s_n^2$ that
\beq{eq:Necessary}
\sum_{i=1}^r\cos nt_i=0\quad\mbox{and}\quad\sum_{i=1}^r \sin nt_i=0,
\eeq
that is, the vectors $(\cos nt_i,\sin nt_i)\in\I R^2$, $i\in [r]$, sum up to zero.

If $f$ differs from $1/r$ on a set of positive measure then,  for at least one integer $n\ge 1$, we have $(c_n,s_n)\not=(0,0)$ and thus \eqref{eq:Necessary} holds. Conversely, if \eqref{eq:Necessary} holds for some $n\ge1$, then we can take, for example, $f(t):=(1+\cos nt)/r$ for $t\in [0,2\pi)$. This  completely describes the set of $r$-tuples in $\SO(2)$ for which the circle $\I S^1$ is ``fractionally'' divisible:

\begin{proposition}\label{pr:Fr}
	An $r$-tuple $(t_1,\ldots,t_r)\in \I T^r$ belongs to $\C F_r$ if and only if~\eqref{eq:Necessary} holds for at least one integer~$n\ge 1$.\qed\end{proposition}

Let us investigate the sets $\Borel_r$ and $\C M'_r$ for $r\le 4$. As we will see, it holds for each $r\le 4$ that $\Borel_r=\C M'_r$ (and, in particular, this set is  also equal to~$\C M_r$).

Let $(t_1,\ldots,t_{r})\in \I T^{r}$. By replacing $(t_1,\ldots,t_r)$ by $(t_1-t_r,\ldots,t_r-t_r)$, which does not affect divisibility, we can assume for convenience that~$t_r=0$. Since $\C M'_r\subseteq \C F_r$, assume that~\eqref{eq:Necessary} holds for some $n\ge 1$. Let $n\ge 1$ be the smallest integer with this property. 

Suppose first that $r=2$. By~\eqref{eq:Necessary} we have $nt_1=(2k+1)\pi$ for some integer $k\ge 0$. Note that $n$ and $2k+1$ are coprime: if an integer $q>1$ divides both $n$ and $2k+1$ then, for $n':=n/q$, we have $n't_1=\frac{2k+1}{q} \pi$ and thus~\eqref{eq:Necessary} holds for $n'<n$, contradicting the minimality of~$n$.
Therefore, the subgroup of $\I T$ generated by $t_1=(2k+1)\pi/n$ is $\left\{\frac{\pi m}{n}\mid m\in \{0,\ldots,2n-1\}\right\}$, which is the additive cyclic group of order $2n$ with $t_1$ corresponding to an odd multiple of the generator~${\pi}/n$. Since the addition of $t_1$ swaps odd and even multiples of ${\pi}/{n}$, 
we have that 
$$
\textstyle A:=\left\{\frac{\pi m}{n}\mid m\in \{0,2,\ldots,2n-2\}\right\}+\left[0,\frac{\pi}n\right)
$$ 
satisfies $t_1+A=[0,2\pi)\setminus A$ and shows that $(r_1,0)\in\Borel_2$, where for $B,C\subseteq \I T$ we denote 
$$
 B+C:=\{b+c\mid b\in B,\ c\in C\}.
 $$
 Thus $\Borel_2=\C M_2=\C M'_2=\Frac_2$ and this set can be equivalently described as consisting of precisely those $(t_1,t_2)\in \I T^2$ such that $t_2-t_1\in\I T$  generates a finite subgroup of even order.

Suppose that $r=3$. Three vectors on the unit circle sum to $\Zero$ if and only if they form an equilateral triangle. (Indeed, the sum of any two unit vectors has norm 1 if and only if the angle between the vectors is $2\pi/3$.)  Thus, up to swapping $t_1$ and $t_2$, we can assume that $nt_1\equiv 2\pi/3$ and $nt_2\equiv 4\pi/3$ modulo $2\pi$.
Each of $t_1,t_2\in [0,2\pi)$ is a (non-zero) integer multiple of $2\pi/(3n)$. Let $k_1,k_2\in [3n-1]$ satisfy $t_i=2\pi k_i/(3n)$.  By the minimality of $n$, the greatest common divisor $\gcd(k_1,k_2,n)=1$. Furthermore, it is impossible that $3$ divides both $k_1$ and $k_2$, for otherwise by e.g.\ $2\pi k_1/(3n)\equiv 2\pi/3\pmod{2\pi}$ we have that $3$ also divides $n$, a contradiction to $\gcd(k_1,k_2,n)=1$. Therefore, the subgroup generated by $t_1,t_2\in\I T$ is $\left\{\frac{2\pi k}{3n}\mid k\in \{0,\ldots,3n-1\}\right\}$, which is the cyclic group of order~$3n$. For $i=1,2$, we have $k_in\equiv in\pmod{3n}$ and thus $k_i\equiv i\pmod 3$. Thus
if we take 
$$
\textstyle A:=\left\{\frac{2\pi m}{3n}\mid m\in \{0,3,\ldots, 3n-3\}\right\}+\left[0,\frac{2\pi}{3n}\right),
$$
then $t_1+A$, $t_2+A$ and $t_3+A=A$ partition~$[0,2\pi)$. 
%(Indeed, we have $t_i.A=\{i,i+3,\ldots,3n-3+i\}+[0,\frac{2\pi}{3n})$ for $i=1,2$.) 
We conclude that $\Borel_3=\C M_3=\C M'_3=\Frac_3$ and this set can be alternatively described as consisting, up to a permutation of indices, precisely of the triples $\left(\frac{2\pi k_1}{3n}+t,\frac{4\pi k_2}{3n}+t,t\right)$ with $n\ge 1$, $k_1,k_2\in [3n-1]$ and $t\in\I T$ such that $\{k_1,k_2\}\equiv \{1,2\}\pmod {3}$ and the greatest common divisor of $k_1$, $k_2$ and $n$  is~1.

Suppose that $r=4$. We need the following geometric claim.

\begin{claim}\label{cl:4Vectors} Four vectors $(x_i,y_i)\in \I S^1$, $i\in [4]$, have sum $\Zero$ if and only if they can be split into two pairs of opposite vectors.\end{claim}

\bcpf 
The non-trivial direction of the claim can be derived by observing that, up to a permutation of indices, we can assume that $\V v:=(x_1,y_1)+(x_2,y_2)$ is a non-zero vector while, in general, there is at most one way to write $-\B v\in\I R^2\setminus\{\Zero\}$ as the unordered sum of two unit vectors. Thus the other two vectors must be $(-x_1,-y_1)$ and $(-x_2,-y_2)$, as desired.\ecpf

Recall that $n\ge 1$ is the smallest integer satisfying~\eqref{eq:Necessary}. Claim~\ref{cl:4Vectors} applied to $x_i:=\cos (nt_i)$ and $y_i:=\sin (nt_i)$ for $i\in [4]$ gives that, up to a permutation of indices, $(x_1,y_1)=-(x_2,y_2)$ and $(x_3,y_3)=-(x_4,y_4)$. Thus, by Proposition~\ref{pr:Fr}, the set $\Frac_4$ consists precisely of
those $(t_1,\ldots,t_4)$ such that, for some integer $n\ge 1$ and up to a permutation of indices, we have that \beq{eq:residue}
n(t_1-t_2)\equiv n(t_3-t_4)\equiv \pi \pmod {2\pi}.
\eeq  

Again, let us assume that $t_4=0$. 

First, let us show that if $t_1/\pi$ is irrational then $(t_1,\ldots,t_4)\not\in\C M'_4$.  By~\eqref{eq:residue}, we can assume that $t_2=t_1+k_2\pi/n$ and $t_3=k_3\pi/n$ for some odd integers $k_2$ and~$k_3$. Suppose for a sake of contradiction that for some measurable subset $A\subseteq \I T$ we have that $\sum_{i=1}^4 t_i.\I 1_A=1$ a.e. Take the Fourier expansion 
$$
\I 1_A(t)=\frac14+\sum_{m=1}^\infty (c_m\cos mt+s_m\sin mt).
$$ 
By the argument leading to~\eqref{eq:Necessary} and Claim~\ref{cl:4Vectors}, we see that $(c_m,s_m)$ can be non-zero only if we can split $(mt_1,\ldots,mt_4)\in\I T^4$ into two pairs, each pair having difference~$\pi$. Since $t_1/\pi$ is irrational, these pairs must be $(t_1,t_2)$ and~$(t_3,t_4)$ by~\eqref{eq:residue}. Thus $mk_i\pi/n\equiv \pi\pmod{2\pi}$ for~$i=2,3$. Clearly, the validity of these two equations is determined by the residue of $m$ modulo $n$. Since $n$ is minimal, these equations cannot both hold for any $m\in[n-1]$. Thus they can hold only if $m$ is a multiple of~$n$. This means that all non-zero Fourier terms of $\I 1_A$ have period $2\pi/n$ as functions $\I T\to\I R$. It follows that 
%the set $A$ is periodic with period $2\pi/n$.
$A= (2\pi k/n)+A$  a.e.\  for every integer~$k$ and $\I 1_A=\frac1n \sum_{k=0}^{n-1} (2\pi k/n).\I 1_A$. Thus 
\begin{eqnarray*}
	t_1.\I 1_A+\I 1_A&=&\frac1n \sum_{k=0}^{n-1} \left( (t_1+2\pi k/n).\I 1_A+(2\pi k/n).\I 1_A\right)\\
	&=&\frac1{2n} \sum_{k=0}^{n-1} (2\pi k/n).\left(t_1.\I 1_A+t_2.\I 1_A+t_3.\I 1_A+t_4.\I 1_A\right)\ =\ \frac12\quad\mbox{a.e.,}
\end{eqnarray*}
where we used that $t_1.\I 1_A+t_2.\I 1_A+t_3.\I 1_A+t_4.\I 1_A=1$ a.e.\ by the choice of $A$. We conclude that the function $2\,\I 1_A$ demonstrates that $(t_1,0)\in\Frac_2$. By the case $r=2$ that was solved earlier, this contradicts the irrationality of~$t_1/\pi$.

This gives that $\C M'_4$ is strictly smaller than $\C F_4$: for example, $(a,a+\pi,\pi,0)$ belongs to $\C F_4\setminus \C M_4'$ if $a/\pi$ is irrational.

Now, suppose that $t_1/\pi$ is rational. Let $\Gamma$ be the subgroup of $\I T$ that is generated by $t_1$, $t_2$ and $t_3$. (There is no need to add $t_4$ as it is 0.) By~\eqref{eq:residue} and the rationality of $t_1/\pi$, the group $\Gamma$ is finite. Of course, if $4$ does not divide its order $|\Gamma|$ then there is no $\B t$-division even if a null set can be removed. So suppose that $|\Gamma|=4m$ for some integer~$m$, i.e.\ that $\Gamma$ is the cyclic group of order~$4m$. 
%(Indeed, if $t\in (0,2\pi)$ is the smallest non-zero element of $\Gamma$ then every other element of is an integer multiple of~$t$.) 
For $i\in [4]$, let $k_i\in \{0,\ldots,4m-1\}$ satisfy that $t_i=\frac{\pi k_i}{2m}$. Let $\B k:=(k_1,\ldots,k_4)$. Let us say that the cyclic group $\I Z_{4m}$, that consists of integer residues modulo $4m$, is \emph{$\B k$-divisible} if there is a subset $A\subseteq \I Z_{4m}$ such that the sets $k_i+A$, $i\in [4]$, partition~$\I Z_{4m}$. Of course, such a set $A$ must have exactly $m$ elements.

The following claim implies in particular that $\Borel_4=\C M_4=\C M'_4$.

\begin{claim}\label{cl:4m} If $\I Z_{4m}$ is $\B k$-divisible then $\B t\in\C B_4$; otherwise, $\B t\not\in \C M'_4$.\end{claim}

\bcpf Suppose first that a subset $A\subseteq \I Z_{4m}$ witnesses the $(k_1,\ldots,k_4)$-divisibi\-li\-ty of~$\I Z_{4m}$. It corresponds to an $m$-subset $B\subseteq [0,2\pi)$ such that its translates by $t_1,\ldots,t_4$ partition the subgroup~$\Gamma\subseteq\I T$. Now the Borel set $C:=B+\left[0,\frac{\pi}{2m}\right)$ exhibits the $\B t$-divisibility of~$\I T$. 

Conversely, suppose that $\I Z_{4m}$ is not $\B k$-divisible. Take any measurable set $C\subseteq [0,2\pi)$ such that its translates by $t_1,\ldots,t_4$ are pairwise disjoint. Take any coset $X:=t+\Gamma\subseteq \I T$ of $\Gamma$. Define $A$ to consist of those $k\in\I Z_{4m}$ such that
$t+\frac{\pi k}{2m}\in C$ (that is, $A$ encodes the intersection of $C$ with the $\Gamma$-coset $X$). The translates of $A$ by $k_1,\ldots,k_4$ in $\I Z_{4m}$ (which correspond to the intersections $(t_i+C)\cap X$, $i\in [4]$) are pairwise disjoint and, by our assumption, omit at least one element of $\I Z_{4m}$. Thus every coset of $\Gamma$ in $\I T$ contains at least one element of
$B:=\I T\setminus (\{t_1,\ldots,t_4\}+C)$. It follows that $B$ has measure at least $2\pi/(4m)$ (as its translates by $\frac{\pi k}{2m}$ for $k\in\{0,\ldots,4m-1\}$ cover~$\I T$). This implies that $\B t\not\in\C M'_4$.\ecpf

Unfortunately, an explicit characterization of the set $\C B_4=\C M_4=\C M'_4$ for general $n$ seems to be rather messy, although it reduces to a finite case analysis for any given $\B t\in\I T^4$ by Claim~\ref{cl:4m}. So we will restrict ourselves to the special cases $n=1$ and $n=2$,
just to illustrate 
%the new phenomenon
% that $\C F_4$ is strictly larger than~$\C M'_4$ (for both $n=1$ and $n=2$) and 
that the measurable $\B t$-divisibility is not determined by the order $4m$ of the group $\Gamma$ alone (which happens already for~$n=2$).

First, assume that $n=1$. 
By~\eqref{eq:residue}, we have up to a permutation that $(t_1,t_2,t_3)\equiv (a,a+\pi,\pi)\pmod {2\pi}$ with $a\not\in \{0,\pi\}$. Thus, working inside $\I Z_{4m}$ (that is, modulo $4m$), we have that $k_2=k_1+2m$ and $k_3=2m$. Since $k_1,k_1+2m,2m$ generate $\I Z_{4m}$, we have that $k_1$ and $2m$ are coprime; in particular $k_1$ is odd. As it is easy to see $A:=\{2i\mid i\in \{0,\ldots,m-1\}\}$ witnesses the
$\B k$-divisibility of~$\I Z_{4m}$. Thus $\B t\in\Borel_4$ by Claim~\ref{cl:4m}. 

Now, assume that $n=2$. By~\eqref{eq:residue}, we have that each of the differences $k_1-k_2$ and $k_3-k_4$ modulo $4m$ is either $m$ or $3m$. We can assume that $k_3=m$ (by negating all $k_i$'s if necessary) and that $k_2=k_1+m$ (by swapping $k_1$ and $k_2$ if necessary). Note that these operations do not affect
the $\B k$-divisibility of~$\I Z_{4m}$ and thus the conclusion of Claim~\ref{cl:4m} is also unaffected. Let $k:=k_1$. Thus
$$
 \B k=(k,k+m,m,0).
$$

First, let us show that if $m=2s$ is even then $\I Z_{4m}$ is $\B k$-divisible (and thus $\B t\in\Borel_4$ by Claim~\ref{cl:4m}). It is enough to find an $s$-set $S\subseteq\{0,\ldots,m-1\}$ such that, modulo $m$, the sets $S$ and $k+S$ partition $\I Z_{m}$ (because then $A:=S\cup(2m+S)$ as a subset of
$\I Z_{4m}$ witnesses the $\B k$-divisibility of~$\I Z_{4m}$). Note that $S:=\{2ik\mid i\in\{0,\ldots,s-1\}\}$ works. (Indeed, by $\gcd(k,m)=1$ each residue modulo $m$ appears exactly once as $ik$ with $i\in \{0,\ldots,m-1\}$ and we have included every second multiple of $k$ into the set $S$.)

Finally, suppose that $m$ is odd. Recall that $\gcd(k,m)=1$. We claim that $\I Z_{4m}$ is $\B k$-divisible if and only if $k\equiv 2\pmod 4$. 

First, suppose that an $m$-set $A\subseteq\I Z_{4m}$ witnesses the $\B k$-divisibility. Since $m$ is odd, some residue $i$ modulo $m$ appears an odd number of times in~$A$. This multiplicity cannot be larger than $2$ since otherwise the translates $k_1+A,\ldots,k_4+A$ would cover the four points $i,i+m,i+2m,i+3m\in\I Z_{4m}$ at least six times. Thus the multiplicity of $i$ in $A$ modulo $m$ is exactly~1. By the commutativity of $\I Z_{4m}$, we can replace $A$ by any its translate. Thus assume that $A$ contains $0$ but none of $m$, $2m$ and~$3m$. 
Thus, by $(k_3,k_4)=(m,0)$, the set $(k_3+A)\cup (k_4+A)$ covers $0$ and $m$ but not $2m$ nor~$3m$. Since $2m\not\in A$, the only way to consistently cover $2m$ and $3m$ is that $2m-k\in A$. Now,  $\{k_1,\ldots,k_4\}+\{0,2m-k\}$ contains $2m-k$ and $3m-k$ but not $-k$ nor~$m-k$. None of the last two elements can be covered by $k_3+A$ or $k_4+A$ (as then $A$ modulo $m$ would contain $-k\pmod m$ at least twice but then the four elements $0,m,2m,3m$ would be covered at least six times, with the extra multiplicity coming from $0$ and $m$ being covered by $0\in A$ when translated by $k_3$ and $k_4$). Thus the only way to consistently cover $-k$ and $m-k$ is that $-2k\in A$. One can continue to argue in this manner, showing that for each $i\in\{0,1,\ldots\}$ we have $-2ik \in A$ and $-(2i+1)k+2m\in A$. 
As the first $m$ of these elements of $A$ are pairwise distinct (in fact, they have pairwise distinct residues modulo $m$) and $m$ is odd, it must hold that the $m$-th element, $-mk+2m$, belongs to $A$. Since $A$ does not contain any of $m$, $2m$ and $3m$, we necessarily have that  $-mk+2m \equiv 0\pmod{4m}$. This equation has $m$ solutions, namely, all $k\in\I Z_{4m}$ with $k\equiv 2\pmod 4$, giving the claim.
 
Conversely, if $k\equiv 2\pmod 4$, then the set $A$ consisting of elements $-2ik$ for $i\in \{0,\ldots,(m-1)/2\}$ and $-(2i+1)k+2m$ for $i\in \{0,\ldots,(m-3)/2\}$ shows the $\B k$-divisibility of~$\I Z_{4m}$. Indeed, note that $|A|=m$ (as its elements have different residues modulo $m$ by $\gcd(k,m)=1$) and that if we keep increasing the index $i$ beyond the stated ranges then we just repeat the elements of $A$ since $-mk+2m\equiv 0\pmod{4m}$. By ``reverse engineering'' the proof of the forward implication, we see that the translates of $A$ by $k_1,\ldots,k_4$ are pairwise disjoint and thus partition~$\I Z_{4m}$, as required.

In the initial version of the manuscript, we conjectured that if $(t_1,\ldots,t_r)\in \C M_r'$ then $(t_i-t_j)/\pi$ is rational for every $i,j\in[r]$. This conjecture was 
subsequently proved by 
Greb{\'\i}k, 
Greenfeld, Rozho\v n and Tao~\cite{GrebikGreenfeldRozhonTao22arxiv}. This implies that $\C B_r=\C M_r=\C M'_r$ for every $r$ (by an argument similar to that of Proposition~\ref{pr:FinOrbit}) and reduces the question if any given $\B t\in \I T^r$ belongs to this set to some finite case analysis.

 \section{Proof of Proposition~\ref{pr:Baire}}
 \label{se:Baire}

In order to prove Proposition~\ref{pr:Baire}, we need some auxiliary results first.
 
 \begin{lemma}\label{lm:CommonEigen} The kernels of real $n_i\times n$ matrices $A_i$, $i\in [k]$, contain a common non-zero vector $\B x\in\I R^n\setminus\{\Zero\}$ if and only if the $n\times n$ 
 	matrix $M:=\sum_{i=1}^k A_i^TA_i$ has zero determinant.\end{lemma}
 
 \begin{proof}  If some non-zero $\B x\in\I R^n$ satisfies $A_i\B x=\Zero$ for every $i\in [k]$, then 
$M\B x=\sum_{i=1}^k A_i^T(A_i\B x)=\Zero$,
% $\B x$ is  a zero eigenvector of the square matrix $M$,
so the determinant of $M$ is zero. 
 
 Conversely, suppose that $M$ is singular. Choose a non-zero vector $\B x\in\I R^n$ with $M\B x=\Zero$. Then
 $$
 0= \B x\cdot M\B x=\sum_{i=1}^k \B x\cdot (A_i^TA_i\B x)= \sum_{i=1}^k (A_i\B x)\cdot (A_i\B x)= \sum_{i=1}^k \|A_i\B x\|_2^2
 $$
 and each $A_i\B x$ must be the zero vector, giving the required.\end{proof}
 
 %Shemesh~\cite{Shemesh84}

 The results of Dekker~\cite{Dekker56}, Deligne and Sullivan~\cite{DeligneSullivan83}, and Borel~\cite{Borel83} (see Theorem 6.4 in~\cite{TomkowiczWagon:btp} and the historical discussion preceding it) give the following.
 
 \begin{lemma}\label{lm:SpecialGamma} For every $d\ge 2$ and $r\ge 2$ there is a choice of rotations $\beta_1,\ldots,\beta_r\in \SO(d)$ that generate
 	the free rank-$r$ group $F_r$ such that its action on $\I S^{d-1}$ is free for even $d$ and \emph{locally commutative} for odd $d$ (meaning that every two elements of $F_r$ that have a common fixed element on $\Sd$ commute).\qed\end{lemma}
 
 Note that the above result is usually stated in the special case $r=2$ as the general case easily follows by taking any subgroup of~$F_2$ isomorphic to~$F_r$.

 \begin{lemma}\label{lm:GenFree} If $\B\gamma=(\gamma_1,\ldots,\gamma_r)\in\SO(d)^r$ is generic, then the rotations $\gamma_1,\ldots,\gamma_r$ generate the free rank-$r$ group $F_r$ and  the corresponding action of $F_r$ on $\Sd$ is free for even $d$ and locally commutative for odd $d$.\end{lemma}

 \begin{proof} For a non-trivial reduced word $w$ in $F_r$ and $\B\beta=(\beta_1,\ldots,\beta_r)\in\SO(d)^r$, the relation $w(\B\beta)=I_d$ amounts to $d^2$ polynomial equations, with $p_{ij}(\B\beta)=0$ stating that the $(i,j)$-th entry of the corresponding product of the matrices of $\beta_i$'s and their transposes (which are equal to their inverses)  is $\I 1_{i=j}$, where $\I 1_{i=j}$ is 1 if $i=j$ and 0 otherwise. Each of these polynomials $p_{ij}$ has rational coefficients. Moreover, the $r$-tuple of matrices $\B\beta$ returned by Lemma~\ref{lm:SpecialGamma} (which, in particular, generates the free subgroup) gives a point where at least one of these polynomials is non-zero, say $p_{ij}(\B\beta)\not=0$. The polynomial $p_{ij}$ has to be non-zero also at the generic point $\B\gamma\in \SO(d)^r$ and so $w(\B\gamma)\not=I_d$. Since $w$ was an arbitrary non-trivial word, the rotations $\gamma_1,\ldots,\gamma_r$ indeed generate the free  group.
 
 Let us show the second part in the case of odd $d$ (with the case of even $d$ being similar). Suppose on the contrary that we have two reduced non-commuting words $w_1$ and $w_2$ in $F_r$ such that the corresponding elements $w_1(\B\gamma)$ and $w_2(\B\gamma)$ have a common fixed point~$\B x\in \Sd$.
 Thus the matrices $A_1:=w_1(\B\gamma)-I_d$ and $A_2:=w_2(\B\gamma)-I_d$ have $\B x\not=\Zero$ as a 
 common zero eigenvector. 
 By Lemma~\ref{lm:CommonEigen}, this property is equivalent to $\det(A_1^TA_1+A_2^TA_2)=0$, which is a polynomial equation in $\B\gamma$ with rational coefficients. 
 For the special $r$-tuple of matrices 
 $\B\beta$ 
 returned by Lemma~\ref{lm:SpecialGamma}, the matrices $B_1:=w_1(\B\beta)-I_d$ and $B_2:=w_2(\B\beta)-I_d$ cannot have  a common zero eigenvector as it would give a common fixed point for the non-commuting elements $w_1(\B\beta)$ and $w_2(\B\beta)$. Thus, we have by Lemma~\ref{lm:CommonEigen} that $\det(B_1^TB_1+B_2^TB_2)\not=0$. We have found a polynomial equality with rational coefficients that holds for $\B\gamma$ but not for $\B\beta\in\SO(d)^r$. This contradicts our assumptions that $\B\gamma\in \SO(d)^r$ is generic.\end{proof}
 
 Also, we will need the following result of Conley, Marks and Unger that directly follows (as a rather special case) from Lemmas~3.4 and~3.6 in~\cite{ConleyMarksUnger20}.
 
 \begin{theorem}[Conley, Marks and Unger~\cite{ConleyMarksUnger20}]
 	\label{th:CMU} Let $F_r$ be the free group of rank $r$ with generators $\gamma_1,\ldots,\gamma_r$ and let $a:F_r\actson X$ be a free Borel action on a Polish space~$X$. Then there is a Borel subset $A\subseteq X$ such that $\gamma_1.A,\ldots,\gamma_r.A$ are disjoint and $X\setminus \cup_{i=1}^r \gamma_i.A$ is meager.\qed
 \end{theorem}

 \begin{proof}[Proof of Proposition~\ref{pr:Baire}.] We have to show that 
 if %$r,d\ge 2$ and 
an $r$-tuple $\B\gamma=(\gamma_1,\ldots,\gamma_r)\in\SO(d)^r$ is generic then there is a $\B\gamma$-division of~$\I S^{d-1}$ with pieces that have the property of Baire.
 
 By Lemma~\ref{lm:GenFree}, the elements $\gamma_1,\ldots,\gamma_r\in\SO(d)$  generate a free (resp.\ locally commutative) action $a$ of the free group $F_r$ on the sphere~$\Sd$ when $d$ is even (resp.\ odd). 
 The more general Corollary 5.12 in \cite{TomkowiczWagon:btp} (which is attributed in \cite{TomkowiczWagon:btp} to Dekker~\cite{Dekker56,Dekker58}) directly gives that $\Sd$ is  $\B\gamma$-divisible, that is, there is a subset $B\subseteq \Sd$ with $\gamma_1.B,\ldots,\gamma_r.B$ partitioning the sphere.

For every $\gamma\in \SO(d)\setminus\{I_d\}$, the set of its fixed points on $\I S^{d-1}$ is closed (as the preimage of $\Zero$ under the continuous map that sends $\B x\in\I S^{d-1}$ to $\gamma.\B x-\B x\in\I R^d$) and has empty relative interior (for otherwise one can choose $d$ linearly independent vectors fixed by $\gamma$, contradicting $\gamma\not=I_d$). In particular, this set is meager. Since the group $F_r$ is countable, the \emph{free part} $X$ of the action $a$ (which consists of $\B x\in\I S^{d-1}$ such that $w.\B x\not=\B x$ for each non-trivial $w\in F_r$) is co-meager. Also, it is easy to show that the free part $X$ is a Borel subset of the sphere (see e.g.~\cite[Lemma~4.4]{Pikhurko21bcc}). 
 
Theorem~\ref{th:CMU}, when applied  to the free action of $F_r$ on $X$, gives a  Borel set $A\subseteq X$ with its translates $\gamma_1.A,\ldots,\gamma_r.A$ being disjoint and $Z:=\I S^{d-1}\setminus \cup_{i=1}^r \gamma_i.A$ being meager. We can additionally assume that $Z$ is $a$-invariant: its saturation $[Z]:=\cup_{w\in F_r} w.Z$ is still meager (since the countable group $F_r$ acts by homeomorphisms) so we can  replace $A$ by  $A\setminus [Z]$ without violating the conclusion of Theorem~\ref{th:CMU}.
 
 Now, we can combine the Borel $\B\gamma$-division of $\I S^{d-1}\setminus Z$ given by Conley, Marks and Unger~\cite{ConleyMarksUnger20} with the $\B\gamma$-division of Dekker~\cite{Dekker56,Dekker58} restricted to~$Z$. Formally, take $C:=A\cup (B\cap Z)$. The set $C$, as the union of a Borel set and a meager set, has the property of Baire while its translates $\gamma_1.C,\ldots,\gamma_r.C$ partition~$\Sd$ 
 by the invariance of~$Z$.\end{proof}

\section{Proof of Lemmas~\ref{lm:GNew} and~\ref{lm:GPChar}}
\label{se:G}
 
This section is dedicated to proving Lemmas~\ref{lm:GNew} and~\ref{lm:GPChar}. Their proofs are rather  technical; this is why we postponed them until the very end.

 \subsection{Some definitions and results from algebraic geometry}
 \label{se:AlgGeo}

 %Zariski topology
 
In this section we present some definitions and results from algebraic geometry that we need. We will follow the notation from the book by Hassett~\cite{Hassett07itag} to which we refer for missing details (and for a nice concrete introduction to most results needed here).

 A field extension $K\hookrightarrow L$ is called \emph{algebraic} if every $x\in L$ is \emph{algebraic} over $K$, that is, satisfies a non-trivial polynomial equation with coefficients in~$K$. Some easy but very useful facts (\ha{Proposition A.16}) are that, for an arbitrary field extension $K\hookrightarrow L$,
 \beq{eq:16.3}
 \mbox{the elements of $L$ that are algebraic over $K$ form a field}
 \eeq
  %(that is, the sum, product or the inverse of algebraic elements is algebraic) 
  and, for another field extension $L\hookrightarrow M$, 
 \beq{eq:16.4}
 \mbox{ 
 	if $K\hookrightarrow L$ and $L\hookrightarrow M$ are both algebraic
 	%\ $\Rightarrow$\ 
 	then $K\hookrightarrow M$ is algebraic.
 }
 \eeq
 %The algebraic closure of $K$ is denoted by $\O K$.
 
 Let us fix a field $K$.
  
 By a \emph{variety} we mean a subset $X$ of some affine space $K^n$ which is closed in the \emph{Zariski topology}, that is,  $X$ is equal to 
 $$
 V_K(\C F):=\{\B x\in K^n\mid \forall\, f\in\C F\ f(\B x)=0\}
 $$ 
 for some family $\C F\subseteq K[\B x]$ of polynomials where $\B x:=(x_1,\ldots,x_n)$. Then the \emph{coordinate ring} of $X$ is $K[X]:=K[\B x]/I(X)$, where 
 $$I(X):=\{f\in K[\B x]\mid \forall\,\B x\in X\ f(\B x)=0\}
 $$
 denotes the \emph{ideal} of the variety~$X\subseteq K^n$.

 We call a variety $X\subseteq K^n$  \emph{irreducible} if we cannot write $X=X_1\cup X_2$ for some varieties $X_1,X_2\subsetneq X$. 
 This is equivalent to the statement that the ideal $I(X)\subseteq K[\B x]$ is prime (\ha{Theorem~6.5}). Then $K[X]$ is a domain so we can define its fraction field, which is called the \emph{function field} of $X$ and is denoted by $K(X)$. Elements of $K[X]$ (resp.\ $K(X)$) can be viewed as the restrictions of polynomial (resp.\ rational) functions to $X$ modulo identifying functions that coincide on $X$.
 
 The \emph{dimension} $\dim X$ of an irreducible variety $X$ is the cardinality of a \emph{transcendence basis} for the field extension $K\hookrightarrow K(X)$, which is a collection of  algebraically independent (over $K$) elements $z_1,\ldots,z_k\in K(X)$ such that $K(X)$ is algebraic over $K(z_1,\ldots,z_k)$, the smallest subfield of $K(X)$ containing $K\cup\{z_1,\ldots,z_k\}$. By \ha{Proposition 7.15}, a transcendence basis exists and every two transcendence bases have the same cardinality.

 Every variety $X$ can be written as a finite union $X_1\cup\ldots\cup X_m$ of irreducible varieties  (\ha{Theorem 6.4}). (In fact, this decomposition, if irredundant, is unique up to a permutation of indices.) Then the \emph{dimension} of $X$ is defined as 
 %the maximum of dimensions of $\dim X_i$, 
 $\dim X:=\max\{\dim X_i\mid i\in [m]\}$. By \cu{Corollary 2.68}, one can equivalently define
 \beq{eq:TopDim}
 \dim X:=\max\{k\mid \mbox{$\exists$ irreducible varieties $Y_1,\ldots Y_k$ with $\emptyset \subsetneq Y_1\subsetneq\ldots\subsetneq Y_k\subseteq X$}\}.
 \eeq
 
 We will also need the following easy result.
 % which can be proved by an induction on $n$ (cf \cu{Theorem 1.4}). 
 
 \begin{lemma}\label{lm:1.4} If $X_1,\ldots,X_n$ are infinite subsets of a field $K$ and a polynomial $f\in K[x_1,\ldots,x_n]$ vanishes on each element of $X_1\times\ldots\times X_n$, then $f$ is the zero polynomial.\end{lemma}
 
 \begin{proof} We use induction on $n$. The base case $n=1$ can be proved by induction on the degree of the univariate polynomial $f(x_1)$ by factoring out a linear factor corresponding to a root of~$f$.
 
 Let $n\ge 2$. Expand $f(x_1,\ldots,x_n)=\sum_{i=0}^m c_i x_n^i$, with $c_i\in K[x_1,\ldots,x_{n-1}]$ and $c_m\not=0$. By induction, there is $(a_1,\ldots,a_{n-1})$ in $X_1\times\ldots \times X_{n-1}$ with $c_m(a_1,\ldots,a_{n-1})\not=0$. Thus $f(a_1,\ldots,a_{n-1},x_n)$ is a non-zero polynomial of $x_n$ so it cannot vanish
 on $X_n$ by the base case $n=1$.\end{proof}

 \subsection{Variety $\SO(d;K)^r$}
 \label{se:SO}

%Let $d\ge 2$ and $r\ge 1$ be integers. 
In this section we show in particular that $\SO(d)^r$, as a variety in $\I R^{d^2r}$, is irreducible and that the set of entries above the diagonals forms a transcendence basis; in particular, the dimension of $\SO(d)^r$ is ${d\choose 2}r$. In fact, we will need an extension of this result, where  the underlying field can be different from $\I R$, for the proof of Lemma~\ref{lm:GPChar} (even though the statement of Lemma~\ref{lm:GPChar} deals only with the real case).

 Let $d\ge1$ be an integer and $K$ be a field.
 %, $\I Q\subseteq k\subseteq \I C$. 
 Consider the affine space $K^{d\times d}$ of all $d\times d$ matrices with entries in $K$, writing its elements as $\gamma=(\gamma_{i,j})_{i,j\in[d]}$. Let the \emph{special orthogonal variety}  over $K$  be the variety  $\SO(d;K):=V_K(I_\SO)\subseteq K^{d\times d}$ defined by the ideal
 \beq{eq:I}
 I_\SO:=\big\langle\, (u_i)_{i\in [d]},\,(f_{ij})_{1\le i<j\le d},\,\det(\gamma)-1\,\big\rangle\subseteq K[\gamma],
 \eeq
 where $u_{i}:=\gamma_{i,1}^2+\ldots+\gamma_{i,d}^2-1$ encodes the fact that each row is a unit vector (when $K\subseteq\I R$), $f_{i,j}:=\gamma_{i,1}\gamma_{j,1}+\ldots+\gamma_{i,d}\gamma_{j,d}$ encodes the orthogonality of the $i$-th and $j$-th rows while the last constraint states that the
 determinant of $\gamma$ is $1$. Note that the ``orthonormality'' constraints force $\gamma$ to have determinant $-1$ or $1$, which follows from
 \beq{eq:det=pm1}
 (\det(\gamma))^2=\det(\gamma^T\gamma)
 %=\det(\gamma^T)\det(\gamma)=
 \equiv \det (I_d)=1\pmod{\langle\, (u_i)_{i\in [d]},\,(f_{ij})_{1\le i<j\le d}\,\rangle}.
 \eeq
 The matrix multiplication makes $\SO(d;K)$ a group.  
 If $K=\I R$ then we get the familiar group $\SO(d)$ of  special  orthogonal real $d\times d$ matrices (and the shorthand $\SO(d)$ will
 always be reserved for the real variety~$\SO(d;\I R)$).
 
Take any integer $r\ge 1$. The $r$-th power $\SO(d;K)^r=\SO(d;K)\times\ldots\times \SO(d;K)$ is a variety in $K^{d^2r}$ since a product of Zariski closed sets is Zariski closed (or since one can write the explicit equations defining~$\SO(d;K)^r$).
 %Its element is an $r$-tuple $(\B\gamma_1,\ldots,\B\gamma_r)$. We will denote the coordinates of $\B\gamma_s$ as $(\ga{s}{ij}$
 
 For $(\gamma_1,\ldots,\gamma_r)\in\SO(d;K)^r$, let 
 $$
 \gamma_U:=(\ga{s}{i,j}\mid s\in [r],\ 1\le i<j\le d),
 $$ 
 be the sequence of the ${d\choose 2}r$ entries strictly above the diagonals. We call these entries \emph{upper}. 
  For notational convenience, we fix an ordering of the coordinates of $K^{d^2r}$ so that all \emph{non-upper} entries  (that is, those on or below the diagonals) come before all upper ones; thus when we write a vector of length $d^2r$
 as $(\B x,\B y)$ then we mean that $\B y$ is the upper part.
 
 \begin{lemma}\label{lm:upper}
 	For every subfield $K\subseteq \I C$, the variety $X:=(\SO(d;K))^r\subseteq K^{rd^2}$ is irreducible, has dimension ${d\choose 2}r$ and the set of upper coordinates forms a transcendence basis of the function field $K(X)$ over $K$.
 \end{lemma}
 
 \begin{proof} First, let us show that $X$ is irreducible The proof of this in the case $r=1$  (for an arbitrary field with $2\not=0$) 
 can be found in~\cite[Proposition~5-2.3]{BoijLaksov08iagmg}. We adopt the argument from~\cite{BoijLaksov08iagmg} to work for any $r\ge 1$. (Note that products need not preserve the irreducibility when the underlying field is not algebraically closed.)

% The set $X$ is easily seen to be closed in the Zariski topology as the product of closed sets. 

 For $\B x\in K^d$ with $\B x\cdot\B x:=\sum_{i=1}^d x_i^2$ non-zero, the map $\rho_{\B x}:K^d\to K^d$ that is defined by
 $$
 \rho_{\B x}(\B y):=\B y -2\,\frac{\B y\cdot \B x}{\B x\cdot\B x}\,\B x,\quad \mbox{for }\B y\in K^d,
 $$
 can be thought of as the reflection of $K^d$ around the hyperplane orthogonal to $\B x$, so we call $\rho_{\B x}$ a \emph{reflection}. 
 Each $\gamma\in\SO(d;K)$ can be written as a product of an even number of reflections, see
 \cite[Proposition~1-9.4]{BoijLaksov08iagmg} (and, conversely, every such product is in~$\SO(d;K)$). In fact, the proof in \cite{BoijLaksov08iagmg}, which proceeds by induction on $d$, shows that at most $m:=2d$ reflections are needed. 
 By inserting the trivial composition $\rho_{\B x} \rho_{\B x}=I_d$ for some $\B x\in K^d$ with $\B x\cdot\B x\not=0$ we can write each $\gamma\in\SO(d;K)$ as the product of exactly $m$ reflections. 
 
Let $U:=\{\B z\in K^{d}\mid \B z\cdot\B z\not=0\}$ and define $f:U^m\to\SO(d;K)$ by 
$$
f(\B z_1,\ldots,\B z_{m}):=\rho_{\B z_{1}}\ldots \rho_{\B z_{m}}\in\SO(d;K),\quad \mbox{for }(\B z_1,\ldots,\B z_{m})\in U^m.
$$
Consider the product map $f^r:(U^{m})^{r}\to \SO(d;K)^r$ that applies $f$ in each of the $r$ coordinates.
As the complement $V:=K^{dm}\setminus U^m$ is Zariski closed (as the finite union over $i\in [m]$ of the sets of $(\B z_1,\ldots,\B z_m)\in K^{dm}$ satisfying the polynomial equation $\B z_i\cdot\B z_i=0$), the complement $W:=K^{dmr}\setminus U^{mr}$ is also Zariski closed as the finite union over $i\in [r]$ of the closed sets $K^{dm(i-1)}\times V\times K^{dm(r-i)}$. Clearly,
$f^r$ is a rational map defined everywhere on $U^{mr}$ and thus continuous in the Zariski topology on $U^{mr}\subseteq K^{dmr}$.
% see https://math.stackexchange.com/questions/350152/the-continuity-of-the-rational-maps-in-the-zariski-topology
 Also, the image of $f^r$ is exactly $X=\SO(d;K)^r$ with the surjectivity following from the choice of~$m$. 
It follows
from \cite[Lemma 5-2.1]{BoijLaksov08iagmg}
that $X$ is irreducible. (In brief, if $X$ can be written as a union of two proper closed subsets $X_1\cup X_2$, then $K^{dmr}$ is a union of two proper closed sets $f^{-1}(X_1)\cup W$ and $f^{-1}(X_2)\cup W$, contradicting the irreducibility of $K^{dmr}$ since its ideal $I(K^{dmr})$, which is $\{0\}$ by e.g.\ Lemma~\ref{lm:1.4}, is trivially prime.) 
Thus $X$ is indeed irreducible.
 
It remains to show that the set of upper coordinates $\gamma_U$ (that is, all entries above the diagonals) is a transcendence basis for the function field $K(X)$ over $K$.
This claim is made of the following two parts.

First, let us show that the field extension $K(\gamma_U)\hookrightarrow K(X)$ is algebraic. By~\eqref{eq:16.3} and~\eqref{eq:16.4}, it is enough to represent this field extension as a composition of field extensions where, at each step, every added non-upper coordinate is algebraic over the previously added 
coordinates and the upper coordinates in the same matrix. Thus we consider just one matrix in $\SO(d;K)$, which we denote as $\gamma=(\gamma_{i,j})_{i,j\in [d]}$. We add the non-upper coordinates by whole rows in the natural order (with Row 1 added first, then Row 2, and so on).
 Take any Row $m$ and a non-upper pair $(m,j)$ (i.e.\ with $j\le m$). 
 The following argument works for every index $j\in [m]$ so we pick $j=m$ for notational convenience. Thus we have to show that $z:=\gamma_{m,m}$, as an element of $K(X)$, is algebraic over 
 $$
 K(\{\gamma_{i,j}: i\in [m-1],\ j\in [d]\}\cup \{\gamma_{m,j}\mid j\in \{m+1,\ldots,d\}\,\}).
 $$
 Let the vector
 $\B x:=(\gamma_{m,1},\ldots,\gamma_{m,m-1})$ consist of the other non-upper entries of Row~$m$ and let $M:=(\gamma_{i,j})_{i,j\in [m-1]}$ be the square submatrix of $\gamma$ which lies above~$\B x$. 
 The orthogonality of Row $m$ to the previous rows gives a  system of $m-1$ linear equations, namely,
 $$M\B x^T=\B f^T,
 $$
  where $\B f:=(f_1,\ldots,f_{m-1})$ with $f_i:=-\gamma_{i,m}z - \sum_{j=m+1}^d \gamma_{i,j}\gamma_{m,j}$ for $i\in[m-1]$.
 %(and we view $f$ and $\B x$ as column vectors). 
 %We see that $f$ is a linear function of $z$ (given the previous coordinates).
 By Cramer's rule, we have $\det(M)\B x^T=\Ad{M}\B f^T$, where $\Ad{M}$ denotes the \emph{adjoint matrix} of~$M$ (whose $(i,j)$-th entry is $(-1)^{i+j}$ times the determinant of $M$ with Row $j$ and Column~$i$ removed). Take the unit ``norm'' relation $\sum_{i=1}^d \gamma_{m,i}^2=1$ for Row $m$, multiply it by $(\det(M))^2$ and replace each $(\det(M))^2x_i^2$ by its value from Cramer's rule. We get a polynomial equation having no $\B x$, namely,
 \beq{eq:om}
 (\det(M))^2z^2+\sum_{i=1}^{m-1} \left(\sum_{j=1}^{m-1} \Ad{M}_{ij} f_j\right)^2+(\det(M))^2\sum_{i=m+1}^d \gamma_{m,i}^2=(\det(M))^2.
 \eeq
 Let us show that the coefficient at $z^2$ in this equation is non-zero. This coefficient is some polynomial in the upper entries and the previous entries. If we take the identity matrix $I_d$ for $\gamma$, then 
 the column above $z$ is all zero and the matrix $M$ is invertible (namely, it is the $(m-1)\times (m-1)$ identity matrix $I_{m-1}$). Then $\B f$ does not depend on $z$ at all and the coefficient at $z^2$ is $(\det(M))^2=1$, which is non-zero.  
So the coefficient at $z^2$ in~\eqref{eq:om} is a non-zero polynomial, that is, $z$ is algebraic over all previous entries, as desired. We conclude (by~\eqref{eq:16.3} and~\eqref{eq:16.4}) that all entries on or below the diagonals are algebraic over~$K(\gamma_U)$ and thus the field extension $K(\gamma_U)\hookrightarrow K(X)$ is indeed algebraic.

Thus in order to show that the coordinates $\gamma_U$ form a transcendence basis, it remains to prove that these ${d\choose 2}r$ coordinates, as elements of the function field $K(X)$, are algebraically independent over~$K$. It is enough to prove this for $K=\I C$. Indeed, we assumed that $K\subseteq \I C$. A non-trivial algebraic relation over $K$ between the upper coordinates means 
%by the Elimination Theorem (\ha{Theorem 4.8}) 
that the ideal that defines $\SO(d;K)^r$ (which, in the case $r=1$, is the ideal $I_{\SO}$ in~\eqref{eq:I})  contains a non-zero polynomial $g$ that does not depend on non-upper coordinates. The same polynomial $g$, when viewed as a polynomial in $\I C[\B \gamma]$, then witnesses that the upper coordinates are algebraically dependent over~$\I C$.

 Thus let us assume that $K=\I C$. We need an easy auxiliary claim first from which we will derive that every choice of sufficiently small in absolute value upper entries can be extended to a matrix in $\SO(d;\I C)$.
 For $m\in [d]$ and an $m\times d$ matrix $\gamma=(\gamma_{i,j})$, let the property $\C P_m$ state that
 for all $i\in [m]$ we have $\sum_{j=1}^d \gamma_{i,j}\gamma_{m,j}=\I 1_{i=m}$. (Recall that $\I 1_{i=m}$ is $1$ if $i=m$ and 0 otherwise.) In other words, $\C P_m$ states that Row $m$ has unit ``norm'' and is orthogonal to all previous rows.
 
 \begin{claim}\label{cl:e} For every $m\in [d]$ and $\delta>0$ there is $\e=\e_m(\delta)>0$ such that the following holds. Take any  complex numbers $(\gamma_{i,j})_{(i,j)\in S}$, where 
 	$$
 	S:=([m-1]\times [d])\cup \{(m,j)\mid m<j\le d\},
 	$$ 
 such that $\C P_1,\ldots,\C P_{m-1}$ hold and $|\gamma_{i,j}-\I 1_{i=j}|\le \e$ for any $(i,j)\in S$. Then there is a choice of $\gamma_{m,1},\ldots,\gamma_{m,m}\in\I C$ such that $|\gamma_{m,j}-\I 1_{m=j}|\le \delta$ for each $j\in [m]$ and $\C P_m$ holds. Moreover, if $\gamma_{i,j}$ for each $(i,j)\in S$ is real then $\gamma_{m,1},\ldots,\gamma_{m,m}$ can additionally be chosen to be real.
 \end{claim}

 \bcpf %Though an explicit $\e$ can be computed, for brevity we present a non-constructive proof.
 Suppose that the claim fails for for  some $m\in [d]$ and $\delta>0$. Let real $\e$ tend to $0$ from above and let 
 $\gamma\in \I C^S$ be a partial assignment violating the claim.
 Let us use the notation that was introduced around~\eqref{eq:om}. By our choice of $\gamma$, we have that
 each entry of $M$ is within additive $\e=o(1)$ from the corresponding entry of the identity matrix and thus $\det(M)=1+o(1)$ is non-zero.
 %, where the asymptotic notation hides terms then tend to 0 with $\e\to 0$. The quadratic equation~\eqref{eq:om} becomes $z^2=1+o(1)$. 
 Of the two roots of the quadratic equation~\eqref{eq:om}, which now reads $z^2-1=o(1+|z|^2)$, choose $z=1+o(1)$.
 In fact,~\eqref{eq:om} gives not only the entry $z=\gamma_{m,m}$ but the  consistent remainder of Row~$m$ by $\B x^T:=(\det (M))^{-1}\,\Ad{M}\B f^T$, satisfying~$\C P_m$. By the continuity of  the all involved functions (and $\det (M)=1+o(1)$), we have $\|\B x\|_\infty=o(1)$, a contradiction to $\delta>0$ being fixed.
 
 Let us show how to adapt this argument to establish the second part of the claim. Suppose additionally that the given $\gamma_{i,j}$'s are reals. 
 In the above notation, the quadratic equation~\eqref{eq:om} has all real coefficients and, as before, states that $z^2-1=o(1+|z^2|)$. Its left-hand side as a function of $z\in\I R$ changes sign at $z=1$ with its derivative~$2z$ being bounded away from 0 around $z=1$. Hence we can choose a real root $z=1+o(1)$. Then $M$ is a real matrix and the rest of Row~$m$, namely $\B x^T:=(\det (M))^{-1}\,\Ad{M}\B f^T$ is also real.\ecpf

 Consider the projection $\pi:\SO(d;\I C)^r\to \I C^{m}$ on the $m:={d\choose 2}r$ upper coordinates, which maps $(\B x,\B y)$ to~$\B y$. In particular, the $r$-tuple of the identity matrices projects to the zero vector $\Zero\in \I C^m$. 
 The image of $\pi$ contains some Euclidean open ball 
 $$\mathrm{Ball}_\e(\Zero):= \{\B z\in \I C^m\mid\|\B z\|_1<\e\}$$ of radius $\e>0$ around the origin. Namely, we can take its radius to be 
 \begin{equation}\label{eq:Epsilon}
 \e:=\e_d(\e_{d-1}(\ldots\e_1(1/(2^d\,d!))\ldots))>0,
 \end{equation}
  where $\e_1,\ldots,\e_d$ are the functions returned
 by Claim~\ref{cl:e}. Indeed, by the choice of the constants we know that for every $\B y\in \mathrm{Ball}_\e(\Zero)$, we can construct a $d\times d$ matrix $\gamma$ row by row so that $\gamma$ projects to $\B y$ and satisfies all properties~$\C P_1,\ldots,\C P_d$ while it also holds that $\|\gamma-I_d\|_\infty<1/(2^d\,d!)$. The last inequality gives, rather roughly, that
 $|\det(\gamma)-1|< 1$. Thus $\det(\gamma)=1$ because  $\det(\gamma)$ is either $-1$ or $1$ by~\eqref{eq:det=pm1}.
 So indeed $\pi(\SO(d;\I C))$ contains~$\pi(\gamma)=\B y$.

Now, suppose on the contrary that there is a non-trivial polynomial relation bet\-ween the upper coordinates.
%By the Elimination Theorem (\ha{Theorem 4.8}),
Thus there is a non-zero polynomial $g$ which does not depend on the non-upper coordinates and belongs to the ideal generated by the polynomials that define $\SO(d;\I C)^r$ (with those for $r=1$ being listed in~\eqref{eq:I}). The polynomial $g$, as a function of the $m$ upper coordinates,  vanishes on $\pi(X)\subseteq \I C^m$. This contradicts Lemma~\ref{lm:1.4} as $\pi(X)$ contains a non-empty open set, namely the open ball of radius $\e$ around the origin, and thus $\pi(X)$ contains a product of $m$ infinite sets.\end{proof}
 
 Now we are ready to show that the set $\Ng$ of non-generic points in $\SO(d)^r$ is ``small''.

\begin{proof}[Proof of Lemma~\ref{lm:GNew}.] 
As before, when we identify an $r$-tuple of $d\times d$ matrices over a field $K$ with an element of $K^{d^2r}$, let us order the $d^2r$ coordinates so that the $m:={d\choose 2}r$ upper entries (i.e.\ those above the diagonals) come at the end. Thus if we write an element of $K^{d^2r}$ as $(\B x,\B y)$ then $\B y$ corresponds to the $m$ upper entries.
Also, we use the standard topology on $\I S^{d-1}$ (the one which is inherited from the Euclidean space~$\I R^d$). 

There are countably many polynomials in $\I Q[\B x,\B y]$ so enumerate those that are non-zero on at least one element of $\SO(d)^r$ as $f_1,f_2,\ldots$~. By definition, if a point $(\B a,\B b)\in \SO(d;\I R)^r$ is not generic then some $f_i$ vanishes on~$(\B a,\B b)$. Thus $\Ng$ is a subset of the countable union $\cup_{i=1}^\infty Z_i$, where 
 \beq{eq:Zi}
 Z_i:=
 \{(\B a,\B b)\in\SO(d)^r\mid f_i(\B a,\B b)=0\}.
 \eeq
% Since the map which sends $(\B x,\B y)$ to $f_i(\B y)$ is 
 %polynomial and 
 Since each polynomial $f_i$ is continuous as a function $\I R^{d^2r}\to\I R$, each set $Z_i$ is closed.

 Let us turn to Part~\ref{it:Gmeas} where we have to show that the Haar measure $\nu$ assigns measure 0 to $\Ng$. 
 By the countable additivity,
 it is enough to show that each set $Z_i$, defined by~\eqref{eq:Zi}, has $\nu$-measure zero.

 First, let us recall how the Haar measure can be constructed for the group $\Gamma:=\SO(n)^r$ (and, in fact, for any real Lie group), following the presentation in \kn{Sections VIII.1--2}. Namely, choose some linear basis for the Lie algebra $(\mathfrak{so}(d))^r$ viewed as the tangent space $T_{(I_d,\ldots,I_d)}$ at the identity $(I_d,\ldots,I_d)\in\SO(d)^r$ and, using the translations of these vectors, turn them into left-invariant vector fields
 $X_1,\ldots,X_m$. (Note the the Lie algebra $(\mathfrak{so}(d))^r$, that consists of all $r$-tuples of skew-symmetric matrices, has dimension $m={d\choose 2}r$ as a vector space.) For each $\B\gamma\in \Gamma$, let $e_1(\B\gamma),\ldots,e_m(\B\gamma)\in T_{\B \gamma}^*$ be the dual basis to $(X_1(\B\gamma),\ldots,X_m(\B\gamma))$. Then $\omega=e_1\wedge\ldots\wedge e_m$ (the skew-symmetric product) is a smooth $m$ form on $\Gamma$, which is positive and left-invariant and thus
 defines a Borel left-invariant non-zero measure on $\Gamma$ (\kn{Theorem 8.21}). By the uniqueness, this has to be a multiple of the Haar measure~$\nu$. In particular, any smooth submanifold of $\Gamma$ of dimension (as a manifold) less than $m$ has zero Haar measure (\kn{Equation~(8.25)}). 
 
 The set $Z_i\subsetneq \SO(d)^r$, as an algebraic variety, has dimension smaller than~$m$ which follows from the definition of the dimension via nested chains of irreducible varieties (that is, by~\eqref{eq:TopDim}) and from the irreducibility of the variety $\SO(d)^r$ (that is, by Lemma~\ref{lm:upper}). Some standard results in the theory of \mbox{(semi-)}\-algebraic sets give 
 that every bounded variety in some $\I R^n$ admits a triangulation into 
 simplices  each of which is a smooth submanifold of $\I R^n$, see e.g.\ \cite[Theorem 5.43]{BasuPollackRoy06arag}. 
Apply this result to every irreducible component $Z\subseteq Z_i$.
The dimension $k$ of each obtained simplex $S$ (as a manifold) is at most~$\dim Z$. Indeed, pick a point $\B s\in S$ and the projection from $S$ on some $k$ coordinates which is a homeomorphism around $\B s$.  Observe that these $k$ coordinates are algebraically independent in the function field $\I R(Z)$  because no non-zero polynomial on $\I R^k$ can vanish on a non-empty open set by Lemma~\ref{lm:1.4}. 

Thus we covered $Z_i$ by finitely many manifolds of dimension less than $m$, each having zero Haar measure as it was observed earlier (by~\kn{Equation~(8.25)}). We conclude that the Haar measure of $Z_i$ is indeed zero.

Let us show Part~\ref{it:GTop}. Recall that the sets $Z_1,Z_2,\ldots$ were defined in~\eqref{eq:Zi}. Clearly, each set $Z_i$ is closed. Thus it is enough to show that the relative interior of each $Z_i\subseteq \SO(d)^r$ is empty. Suppose on the contrary that the relative interior $U$ of some $Z_i$ is non-empty. Since the compact group $\SO(d)^r$ acts transitively on itself by homeomorphisms, finitely many translates of $U$ cover the whole group. As the Haar measure is $\nu$ is invariant under this action, we have that $\nu(U)>0$. However, this contradicts Part~\ref{it:Gmeas} that we have already proved.
 
 This finishes the proof of Lemma~\ref{lm:GNew}.\end{proof}

% \section{Generic points}

\subsection{Proof of Lemma~\ref{lm:GPChar}}
\label{se:aux}

%Recall that Lemma~\ref{lm:G} claims that the set $\Ng$ of non-generic points of $\SO(d)^r$ is ``small'' in various senses. 
Our proof of the reverse (harder) implication of Lemma~\ref{lm:GPChar}
%its Item~\ref{it:GAlg} 
needs  Lemma~\ref{lm:suff} below. Since we could not find this rather natural statement anywhere in the literature we present a proof whose main idea (to use dimension) was suggested to us by Miles Reid. In fact, Miles Reid came up with a full proof of some initial version of the lemma. Since his proof relies on the so-called \emph{universal domain} of $K$ while we would like to have this paper as elementary as possible, we present a proof that avoids universal domains.

Given a field extension $K\hookrightarrow L$ and a variety $X\subseteq L^n$ (over the field $L$), we say that an element $\B a\in X$ is \emph{$K$-generic} for $X$ if every polynomial $p\in K[x_1,\ldots,x_n]$ with $p(\B a)=0$  vanishes on every element of~$X$.
(Here as well as in the rest of this paper, each evaluation mixing elements of some two fields $K\hookrightarrow L$ is done in the larger field~$L$.)
In the special case when $K:=\I Q$, $L:=\I R$, $X:=\SO(d)^r$ we get exactly the definition of a generic $r$-tuple of rotations from the Introduction.

 \begin{lemma}\label{lm:suff}
	Let $K\hookrightarrow L$ be a field extension, with $L$ being algebraically closed.
	% of characteristic 0. 
	Let $\C P\subseteq K[\B x,\B y]$ be some family of polynomials over $K$, where we abbreviate $\B x:=(x_1,\ldots,x_m)$ and $\B y:=(y_1,\ldots,y_n)$. Suppose that
	\beq{eq:X}
	X:=
	%V_L(f_1,\ldots,f_r):=
	\{(\B x,\B y)\in L^{m+n}\mid \forall\ f\in \C P\ f(\B x,\B y)=0\},
	\eeq
	as a variety over $L$,  is irreducible and has dimension $n$ with $y_{1},\ldots,y_n$ forming a transcendence basis for the function field $L(X)$ over $L$.

	Then every $p=(\B a,\B b)\in X$ with the $n$-tuple $\B b\in L^n$ being algebraically independent over $K$ is a {$K$-generic point} of $X$. 
	%(that is, every polynomial $f\in K[\B x,\B y]$ with $f(p)=0$ vanishes on the whole set $X$).
\end{lemma}

\begin{proof}
Let the ideal $I_p\subseteq K[\B x,\B y]$ consist of those polynomials over $K$ that vanish on~$p$. Let
$$
Z:=V_L(I_p)=\{(\B x,\B y)\in L^{m+n}\mid \forall\, f\in I_p\ f(\B x,\B y)=0\}.
$$
As $\C P\subseteq I_p$, we trivially have that $Z\subseteq X$. We have to show that $Z=X$, which by the definition of $Z=V_L(I_p)$ will give the required result (namely, that every $f\in I_p$ vanishes on $X$).

Let $Z=Z_1\cup\ldots\cup Z_t$ be a decomposition of $Z$ into irreducible varieties (\ha{Theorem 6.4}). 

Suppose first that there is $i\in [t]$ such that the $n$-tuple $\B y$, with each $y_j$ viewed as an element of the function field $L(Z_i)$,  is algebraically independent over~$L$. This means that the dimension of the irreducible variety $Z_i\subseteq L^{m+n}$ is at least $n$. Recall that $Z_i\subseteq Z\subseteq X$. By the definition of the dimension via nested chains of irreducible subvarieties (that is, by~\eqref{eq:TopDim}), we cannot have $Z_i\subsetneq X$ for otherwise any chain for $Z_i$ extends to a strictly larger chain for $X$ which gives that
$\dim X-1\ge \dim Z_i\ge n$, contradicting our assumption.
Thus $Z_i=Z=X$, as desired.

Thus we can assume that for every $i\in [t]$ there is a non-zero $g_i\in L[\B y]\cap I(Z_i)$. 
Since $Z=\cup_{i=1}^t Z_i$, we have by \ha{Proposition~3.12} that $I(Z)=\cap_{i=1}^t I(Z_i)$.
(Recall that, for example, the ideal $I(Z)$ of $Z\subseteq L^{m+n}$ consists of those $p\in L[\B x,\B y]$ that vanish on~$Z$.) 
Thus the product $g_1\ldots g_t\in L[\B y]$, which trivially belongs to each $I(Z_i)$, also belongs to~$I(Z)$.

Let $I_{p}^L$ be the ideal in $L[\B x,\B y]$ generated by $I_p\subseteq K[\B x,\B y]\subseteq L[\B x,\B y]$.  In other words, 
$$
 I_p^L:=\left\{\sum_{i=1}^m h_i(\B x,\B y)f_i(\B x,\B y)\mid m\ge 0,\ h_1,\ldots,h_m\in L[\B x,\B y],\ f_1,\ldots,f_m\in I_p\right\},
 $$
  from which it easily follows that $V_L(I_p^L)=V_L(I_p)=Z$.
Since $L$ is algebraically closed, we have by Hilbert's Nullstellensatz~(\ha{Theorem 7.3}) that $I(Z)$ is equal to 
$$
\sqrt{I_p^L}:=\{f\in L[\B x,\B y]\mid \exists\, N\ f^N\in I_p^L\},
$$ 
the \emph{radical} of~$I_p^L$.
Thus there is some integer $N\ge 1$ such that $g:=(g_1\ldots g_t)^N$ belongs to~$I_p^L$. 

In other words, we have shown that $I_p^L$ contains a non-zero polynomial $g$ that does not depend  on $\B x$, that is,
\beq{eq:IpLy}
I_p^L\cap L[\B y]\not=\{0\}.
\eeq
We claim that, in fact,  $I_p\cap K[\B y]\not=\{0\}$. In order to show this, we analyse how a known algorithm for eliminating variables works, arguing that we can run two instances of the algorithm, one for $I_p^L\cap L[\B y]$ and the other for $I_p\cap K[\B y]$, to produce the same generating set of polynomials in each case.

Since all following steps are fairly standard, we will be rather brief, referring the reader to \cite{Hassett07itag} for a detailed exposition. First, by the Hilbert Basis Theorem (\ha{Corollary 2.22}), there is a finite set $\C F\subseteq K[\B x,\B y]$ that generates $I_p$. Of course, the same set $\C F$, as a subset of $L[\B x,\B y]$, generates $I_p^L$. 
We fix any monomial order $\prec$ for $(\B x,\B y)$ which is an elimination order for $\B x$ (\ha{Definition 4.6}) and apply Buchberger's algorithm (\ha{Corollary 2.29}) to find a $\prec$-Gr\"obner basis $\C G$ for~$I_p^L$ using $\C F$ as its input. At a very low level, each step of the algorithm is to pick some two previous non-zero polynomials $h_1$ and $h_2$, take the coefficients $c_1$ and $c_2$ at their $\prec$-highest monomials and add $h_1-(c_1/c_2)h h_2$ for some monomial $h$ to the current pool of polynomials.
Thus all encountered polynomials have coefficients in $K$; in particular,
the obtained Gr\"obner basis $\C G$ is a subset of $K[\B x,\B y]$. By 
the Elimination Theorem (\ha{Theorem 4.8}) and 
our choice of the monomial order~$\prec$, 
%(and the correctness of Buchberger's algorithm), 
the ideal $I_p^L\cap L[\B y]$ is generated by $\C G\cap L[\B y]$, that is, by those polynomials in $\C G$ that
do not depend on $\B x$. 
%By what we previously observed, $\C G\cap L[\B y]=\C G\cap K[\B y]$.
Moreover,  if we apply Buchberger's algorithm to find the intersection of $I_p=\langle \C F\rangle\subseteq K[\B x,\B y]$ and $K[\B y]$, we obtain the very same generating set $\C G\cap K[\B y]$
(because the choice of $h_1$, $h_2$ and $h$ at each low-level step of the algorithm depends only on the $\prec$-highest monomials of the previous polynomials).  

However, we know that $I_p\cap K[\B y]=\{0\}$ because no non-zero polynomial in $K[\B y]$ can vanish on $p$ by our assumption that $\B y$ is algebraically independent over $K$. Thus $\C G\cap L[\B y]=\C G\cap K[\B y]$ can contain only the zero polynomial. This means that $I_p^L\cap L[\B y]=\{0\}$, contradicting~\eqref{eq:IpLy} and proving the lemma.\end{proof}

Now we are ready to prove  Lemma~\ref{lm:GPChar} that gives an alternative characterisation of $\I Q$-generic points of $\SO(d)^r$.
% where we have to show that the set of non-generic points of $\SO(d)^r$ is ``small'' in various senses.

\begin{proof}[Proof of Lemma~\ref{lm:GPChar}.] As before, the $m:={d\choose 2}r$ upper entries of $\SO(d)^r\subseteq \I R^{d^2r}$ come at the end and if we write an element of $K^{d^2r}$ as $(\B x,\B y)$ then $\B y$ corresponds to the $m$ upper entries.
% also, let $\pi$ be the projection $(\B x,\B y)\mapsto \B y$.

The forward implication of the lemma is easy.
Take any $(\B a,\B b)\in SO(d;\I R)^r$ such that $f(\B b)=0$ for some non-zero polynomial $f$ with rational coefficients. Take any vector $\B b'\in \I R^m$ whose $L^\infty$-norm is at most the expression in~\eqref{eq:Epsilon} with entries algebraically independent over~$\I Q$. By Claim~\ref{cl:e}, there is a choice of a real vector $\B a'$ with $(\B a',\B b')\in\SO(d)^r$, that is, we can extent the vector $\B b'$ of upper entries to an $r$-tuple of real special orthogonal matrices. Since the polynomial $f$ with rational coefficients cannot vanish on $\B b'$, the polynomial map $(\B x,\B y)\mapsto f(\B y)$ shows that $(\B a,\B b)$ is not a generic point.

%Let us show Part~\ref{it:GAlg}. 
Let us show the converse implication.
% of Lemma~\ref{lm:GPChar}. 
Let $(\B a,\B b)\in SO(d;\I R)^r$ be any point with the $m$-tuple $\B b\in\I R^m$ of reals being algebraically independent over~$\I Q$.

By Lemma~\ref{lm:upper}, the complex variety $X:=\SO(d;\I C)^r\subseteq \I C^{d^2r}$ is irreducible and the upper coordinates $\B y$ form a transcendence basis for the function field $\I C(X)$. Now, Lemma~\ref{lm:suff} (which requires that the field $L$ is algebraically closed) applies with $K:=\I Q$, $L:=\I C$ and $\C P\subseteq \I Q[\B x,\B y]$ consisting of the polynomials that define the variety $\SO(d;\I R)^r$ (with the ones in~\eqref{eq:I} corresponding to the case $r=1$). The lemma gives that $(\B a,\B b)\in\SO(d;\I R)^r\subseteq \SO(d;\I C)^r$ is a $\I Q$-generic point of~$\SO(d;\I C)^r$. Of course, this trivially implies that $(\B a,\B b)$ is a $\I Q$-generic point also of $\SO(d;\I R)^r$ (as every polynomial $p\in \I Q[\B x,\B y]$ that vanishes on $(\B a,\B b)$ has to vanish on $\SO(d;\I C)^r\supseteq \SO(d;\I R)^r$), as desired.

This finishes the proof of Lemma~\ref{lm:GPChar}.\end{proof}

%------
% Insert acknowledgments and information
% regarding funding at the end of the last
% section, i.e., right before the bibliography.
%------

\begin{ack}
The authors would like to thank Miles Reid for the very useful discussions and for sharing his proof of some initial version of Lemma~\ref{lm:suff}. Also, the authors thank the anonymous referee for helpful comments.
\end{ack}

\begin{funding}
Clinton T.\ Conley was
supported by NSF Grant DMS-1855579. Jan Greb\'\i k was supported by 
Leverhulme Research Project Grant RPG-2018-424. Oleg Pikhurko was supported by ERC Advanced Grant 101020255 and
Leverhulme Research Project Grant RPG-2018-424.
\end{funding}

%------
% Insert the bibliography.
%------

%\begin{thebibliography}{99}

%------ Example for a paper in journal:
% \bibitem{article1}
% \textsc{A.~Petrunin}, Parallel transportation for Alexandrov space with curvature bounded below.
% \emph{Geom. Funct. Anal.} \textbf{8} (1998), no.~1, 123--148.
% \Zbl{0903.53045} \MR{1601854}

%------ Example for a book:
% \bibitem{book1}
% \textsc{W.~P. Ziemer}, \emph{Weakly differentiable functions}.
% Grad. Texts in Math. 120, Springer, New York, 1989.
% \Zbl{0692.46022} \MR{1014685}

%------ Example for a paper in a book:
% \bibitem{incollection1}
% \textsc{J.~S. Milne}, Introduction to Shimura varieties.
% In \emph{Harmonic analysis, the trace formula, and Shimura varieties},
% edited by M.~W. Marcellin and E.~Giorgi, pp. 265--378, Clay Math. Proc. 4,
% Amer. Math. Soc., Providence, RI, 2005. \Zbl{1148.14011} \MR{2192012}

%------ Example for a preprint on arXiv:
% \bibitem{preprint1}
% \textsc{D.~V. Nguyen, S.~K. Chilappagari, M.~W. Marcellin \textup{and} B.~Vasic},
% LDPC codes from latin squares free of small trapping sets. 2010, \arxiv{1008.4177}.

%------ Example for a report:
% \bibitem{report1}
% \textsc{J.~Schöberl}, Commuting quasi-interpolation operators.
% Technical report isc-01-10-math, Texas A\&M University, 2001,
% \url{www.isc.tamu.edu/publications-reports/tr/0110.pdf}.

%------ Example for a thesis:
% \bibitem{thesis1}
% \textsc{E.~Giorgi}, \emph{The geometric universe}.
% Ph.D. thesis, University of Maryland, College Park, 2002.

%\end{thebibliography}

%\bibliography{oleg,sets,misc,ramsey,enum,number,posets,sat,ex,matroid,design,random,graph,general,geometry,algorithm,Analysis,limits}

\end{document}